%% file: main.tex
\documentclass[oneside]{amsart}

\usepackage{graphicx}
\usepackage{amsmath, amssymb, amsthm, upgreek}
\usepackage{mathtools}
\usepackage{enumitem}
\usepackage{stmaryrd}
\usepackage{float}
\usepackage{caption}
\usepackage{subcaption}
\usepackage{import}
\usepackage[most]{tcolorbox}
\usepackage{etoolbox}
\usepackage{csquotes}
\usepackage{hyperref}

\theoremstyle{plain}
\newtheorem{theorem}{Theorem}
\newtheorem{proposition}[theorem]{Proposition}
\newtheorem{lemma}[theorem]{Lemma}

\theoremstyle{remark}
\newtheorem{remark}[theorem]{Remark}
\newtheorem*{remark*}{Remark}
\theoremstyle{definition}
\newtheorem{definition}[theorem]{Definition}
\theoremstyle{definition}

\numberwithin{equation}{section}
\numberwithin{theorem}{section}

\DeclareMathOperator{\arcsinh}{arcsinh} 
\newcommand*{\R}{\mathbb{R}} 
\renewcommand*{\d}{\mathrm{d}} 
\newcommand*{\dt}{\frac{\mathrm{d}}{\mathrm{d}t}} 
\newcommand{\mr}[1]{\mathrm{#1}}
\newcommand{\ols}[1]{\mskip.5\thinmuskip\overline{\mskip-.1\thinmuskip {#1} \mskip-.4\thinmuskip}\mskip.5\thinmuskip}

\title{Delaunay surfaces with free boundary on the unit $2$-sphere}

\author[E.~Lang]{Eric Lang}
\email{eric.lang@uni-bonn.de}

\author[C.~Scharrer]{Christian Scharrer}
\email{scharrer@iam.uni-bonn.de}

\address{Institute for Applied Mathematics, University of Bonn, Endenicher Allee 60, 53115 Bonn, Germany}

\begin{document}

\maketitle

\begin{abstract}
	We characterize all symmetric, embedded surfaces of revolution with constant mean curvature that meet the unit sphere orthogonally. It is well known that there exists a unique free boundary catenoid inside of the unit ball. By analogy, we prove the existence of a unique free boundary nodoid inside of the unit ball with constant mean curvature equal to $-1$. Moreover, there exists a unique compact free boundary nodoid \emph{outside} of the unit ball with constant mean curvature equal to $-1$. However, there exists no symmetric surface of revolution with constant mean curvature $1$ that meets the unit sphere orthogonally.  
\end{abstract}

\input{Contents/intro.tex}
\input{Contents/roulettes.tex}

\input{Contents/Delaunay.tex}

\input{Contents/intersections.tex}

\section*{Acknowledgments}
Parts of this article were submitted as the first author's bachelor thesis at the University of Bonn. The authors want to thank Stefan Müller for his interest in the work and for reviewing the thesis.

\bibliographystyle{abbrv}
\bibliography{mybib}
\end{document}

%% file: Contents/intro.tex
\section{Introduction}
\subsection{Free boundary minimal surfaces}\label{sec:intro:fb_minimal}
Let $(M^2,g_0)$ be a connected, compact $2$-dimensional Riemannian manifold with $\partial M\neq\varnothing$. In 1902, \textsc{Steklov}~\cite{Steklov1902} studied the eigenvalue problem for the Laplace--Beltrami operator $\Delta_{g_0}$ of finding a pair $u\in C^\infty(M)$ and $\sigma\in\R$ such that
\begin{equation*}
	\begin{cases*}
		\Delta_{g_0}u=0&\text{in $M$}\\
		\frac{\partial u}{\partial\eta}=\sigma u&\text{on $\partial M$}
	\end{cases*}
\end{equation*}
where $\eta=\eta(g_0)$ is the outer unit normal to $\partial M$. This problem has a discrete set 
\begin{equation*}
	0=\sigma_0<\sigma_1(g_0)\le\sigma_2(g_0)\le\ldots
\end{equation*}
of eigenvalues. We denote the family of smooth bilinear $2$-forms on $M$ by $S^2(M)$ so that the set of Riemannian metrics is given by $\{g\in S^2(M)\colon g>0\}$. Suppose $g_0$ maximizes the \emph{first Steklov eigenvalue}. That is, $g_0$ maximizes the scale-invariant product
\begin{equation*}
	\sigma_1(g_0)\,\ell_{g_0}(\partial M)=\sup_{\substack{g\in S^2(M)\\ g>0}}\sigma_1(g)\,\ell_{g}(\partial M),
\end{equation*}
where $\ell_g(\partial M)$ denotes the length of $(\partial M,g)$. By a result of \textsc{Fraser--Schoen}~\cite{FraserSchoen16}, after rescaling $g_0$ to $\sigma_1(g_0)=1$, there exists an implicit number $n\ge2$ and independent first Steklov eigenfunctions $u_1,\ldots,u_n$ such that
\begin{equation*}
	u\vcentcolon=(u_1,\ldots,u_n)\colon M\to \ols{\mathbb B}^n
\end{equation*} 
realizes a conformal, branched isometric minimal immersion of $M$ into the unit ball $\ols{\mathbb B}^n=\{x\in\R^n\colon |x|\le 1\}$ such that $\Sigma=u(M)$ meets $\mathbb S^{n-1}=\partial \ols{\mathbb B}^n$ orthogonally. This motivates the study of minimal surfaces in $\ols{\mathbb B}^3$ that meet $\partial \ols{\mathbb B}^3$ orthogonally. A surface meeting the boundary of its container manifold orthogonally is commonly referred to as \emph{free boundary} surface. The most obvious such minimal surface in~$\ols{\mathbb B}^3$ is the unit disk. 
The only other rotationally symmetric, free boundary minimal surface in $\ols{\mathbb B}^3$ is a suitably scaled catenoid, see Section~\ref{sec:fb_catenoid} and Figure~\ref{fig:fb_catenoid}.
    
\subsection{Free boundary Delaunay surfaces}
Minimal surfaces are the critical points of the area functional. They are characterized by having vanishing scalar mean curvature $H=0$ where $H$ is the arithmetic mean of the principle curvatures. In fact, the unit normal multiplied by the sum of the principle curvatures is the $L^2$-gradient of the area functional. For closed surfaces, the $L^2$-gradient of the volume functional is given by the outer unit normal itself. Thus, surfaces of prescribed volume that are critical for the area functional are surfaces of constant scalar mean curvature  
\begin{equation}\label{eq:intro:H}
	H=\lambda
\end{equation}
where $\lambda\in\R$ plays the role of a Lagrange multiplier for the volume functional. The goal of this article is to study rotationally symmetric surfaces of constant mean curvature that meet the unit sphere orthogonally. Apart from generalizing the problem of finding free boundary minimal surfaces described in Section~\ref{sec:intro:fb_minimal}, this is motivated by a model of soap films. Indeed, one may consider a spherical glass bulb with a soap film attached. Just like in Plateau's problem, the soap film minimizes its surface area. It traps an unchangeable amount of air against the glass bulb while its boundary can move freely---an experimental example of a free boundary constant mean curvature surface.

\subsubsection{Rotationally symmetric surfaces of constant mean curvature}
Let $I$ be an interval and $\rho\in C^2(I)$. For the surface in $\R^3$ obtained by rotating the graph $x\mapsto(x,\rho(x))$ around the $x$-axis, the second-order partial differential equation \eqref{eq:intro:H} reads
\begin{equation*}
	\frac{1}{2}\left(\frac{-\ddot\rho}{(1+\dot\rho^2)^\frac{3}{2}}+\frac{1}{\rho(1+\dot\rho^2)^\frac{1}{2}}\right)=\lambda. 
\end{equation*}
This second-order differential equation can be integrated twice, leading to a $3$-parameter family of constant mean curvature surfaces where one parameter is given by $\lambda$ and one parameter corresponds to translation along the axis of revolution. Indeed, for $\lambda\neq0$, integrating once gives the separable equation
\begin{equation*}
	\rho^2-\frac{2\rho}{\lambda\sqrt{1+\dot\rho^2}}=c
\end{equation*}
where $c\in\R$ is the first integration parameter. \textsc{Delaunay}~\cite{Delaunay1841JMPA} observed that its solutions are precisely the roulettes of the conics, see Section \ref{sec:roulettes}. The corresponding surfaces of revolution are referred to as \emph{nodoids} for negative mean curvature and \emph{unduloids} for positive mean curvature. We will only be interested in symmetric free boundary \emph{Delaunay surfaces}. That is, $I$ contains the origin and, for all $x\in I$,
\begin{equation}\label{eq:intro:rho}
	\rho(x)=\rho(-x).
\end{equation} 
Under this condition, we no longer have the freedom of translation along the axis of rotation, leaving us with a $2$-parameter family where we only need to distinguish two cases; one where $\rho$ attains a maximum at the origin, and one where $\rho$ attains a minimum at the origin. 

\subsubsection{Free boundary nodoids outside of the unit ball}
We denote by $N^-(a,b)$ with $a,b>0$ the $2$-parameter family of profile curves whose corresponding surfaces of revolution have constant mean curvature $\lambda=-\frac{1}{2a}$ where the local graph representation~\eqref{eq:intro:rho} attains a maximum at $\rho(0)=a+\sqrt{a^2+b^2}$, see Equation~\eqref{eq:N} and Figure~\ref{fig:Nodary1}. Those nodoids with free boundary on the unit sphere can be characterized as the zeros of the function $G^-$ defined in Proposition~\ref{prp:nodary:intersection}. Given $a,t>0$, there exists a unique $b>0$ such that $N^-(a,b)$ restricted to $(-t,t)$ lies outside of the unit disk and meets the unit circle orthogonally at $t$ if and only if $G^-(a,t)=0$. There holds $\partial_aG^->0$ and $\partial_t^2G^->0$ (on the relevant subdomain), see Lemma~\ref{lem:nodoid:minus}. Thus, there exist precisely two branches of free boundary nodoids outside the unit ball parametrized by $a$ in corresponding intervals $I_1,I_2\subset(0,\infty)$, see Theorem~\ref{thm:nodoid:minus}. One branch with $I_1=(0,\beta]$ for some $\frac{1}{2}<\beta<1$ emanates from the north pole, see Figure~\ref{fig:fb_nodoid:branch_1}. The second branch with $I_2=(\frac{1}{2},\beta]$ emanates from the upper unit circle which itself represents a surface of constant mean curvature $1$, see Figure~\ref{fig:fb_nodoid:branch_1_and_2}. The two branches meet at $a=\beta$ (see Figure~\ref{fig:fb_nodoid:branch_1_and_2_border}) after which, for $a>\beta$, there do not exist any further free boundary nodoids. In particular, there exists a unique free boundary nodoid outside the unit ball with $H=-1$, see Figure~\ref{fig:fb_nodoid:branch_1_border_2}.

\begin{figure}[htbp]
	\centering
	\begin{subfigure}[b]{0.45\textwidth}
		\centering
		\includegraphics[width=\textwidth]{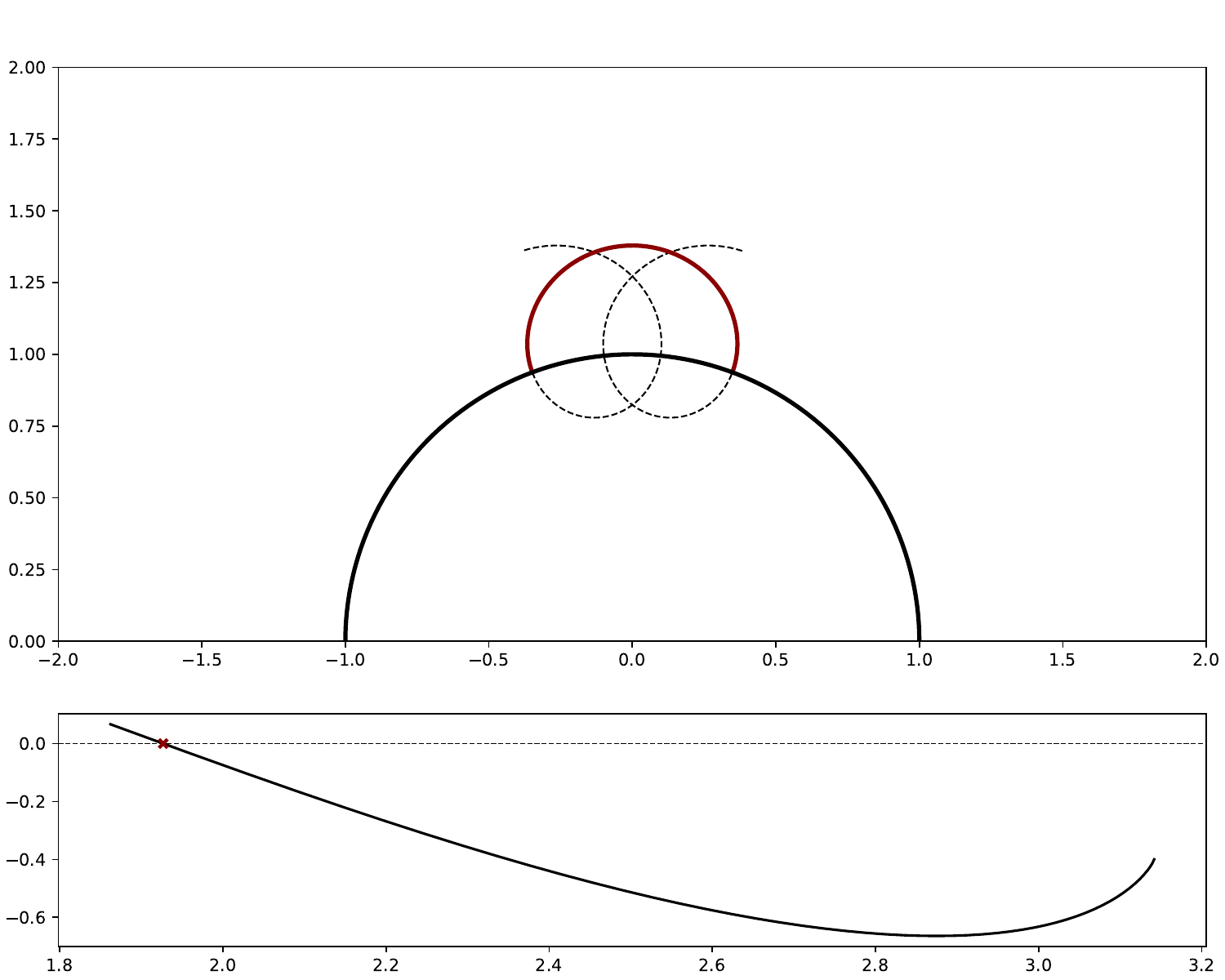}
		\caption{$N^-(a,b)$ with $a=0.3$, $b\approx1.0$}
		\label{fig:fb_nodoid:branch_1}
	\end{subfigure}
	\hfill 
	\begin{subfigure}[b]{0.45\textwidth}
		\centering
		\includegraphics[width=\textwidth]{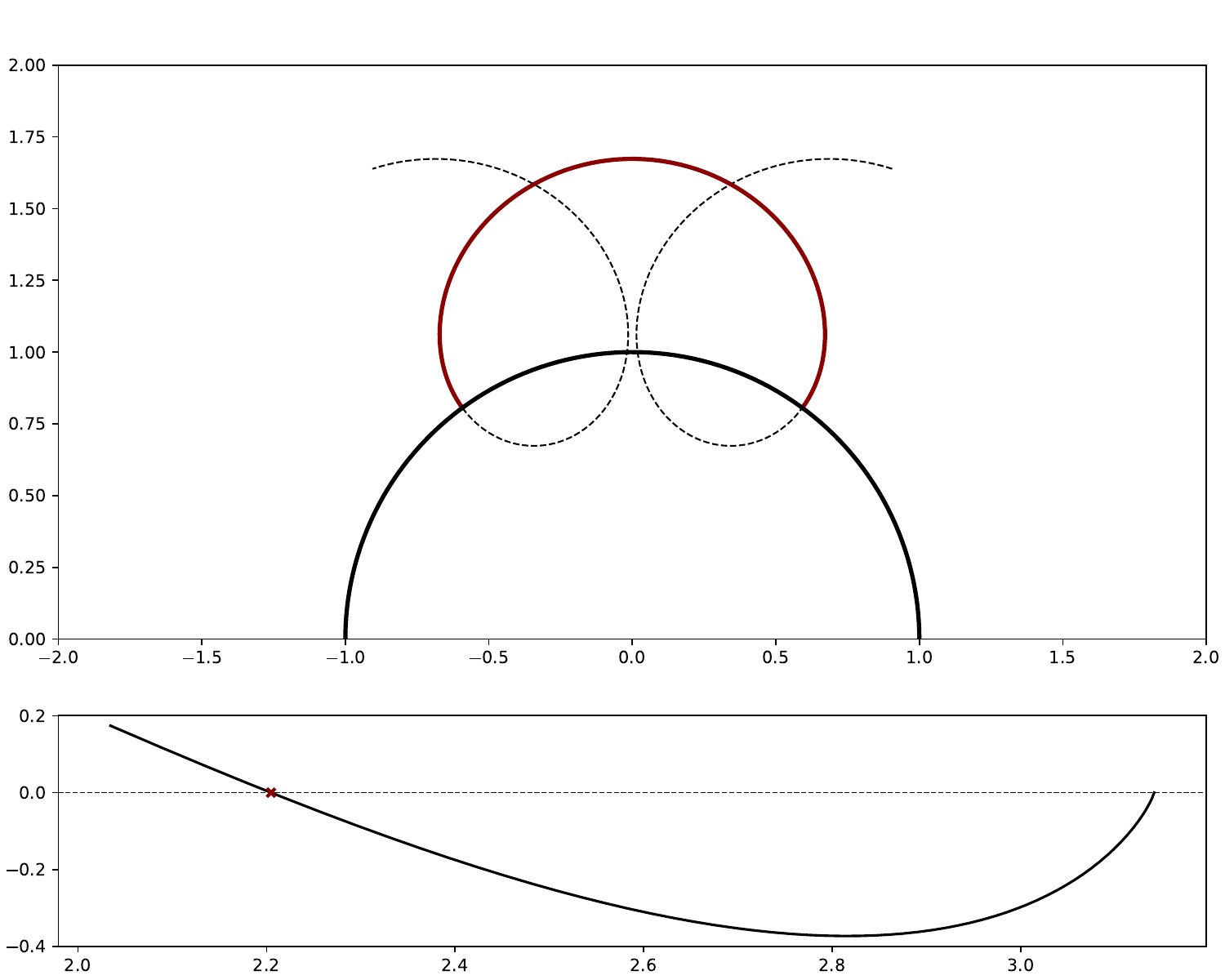}
		\caption{$N^-(a,b)$ with $a=0.5$, $b\approx1.1$}
		\label{fig:fb_nodoid:branch_1_border_2}
	\end{subfigure}
	\vspace{0.5cm} 
	\begin{subfigure}[b]{0.45\textwidth}
		\centering
		\includegraphics[width=\textwidth]{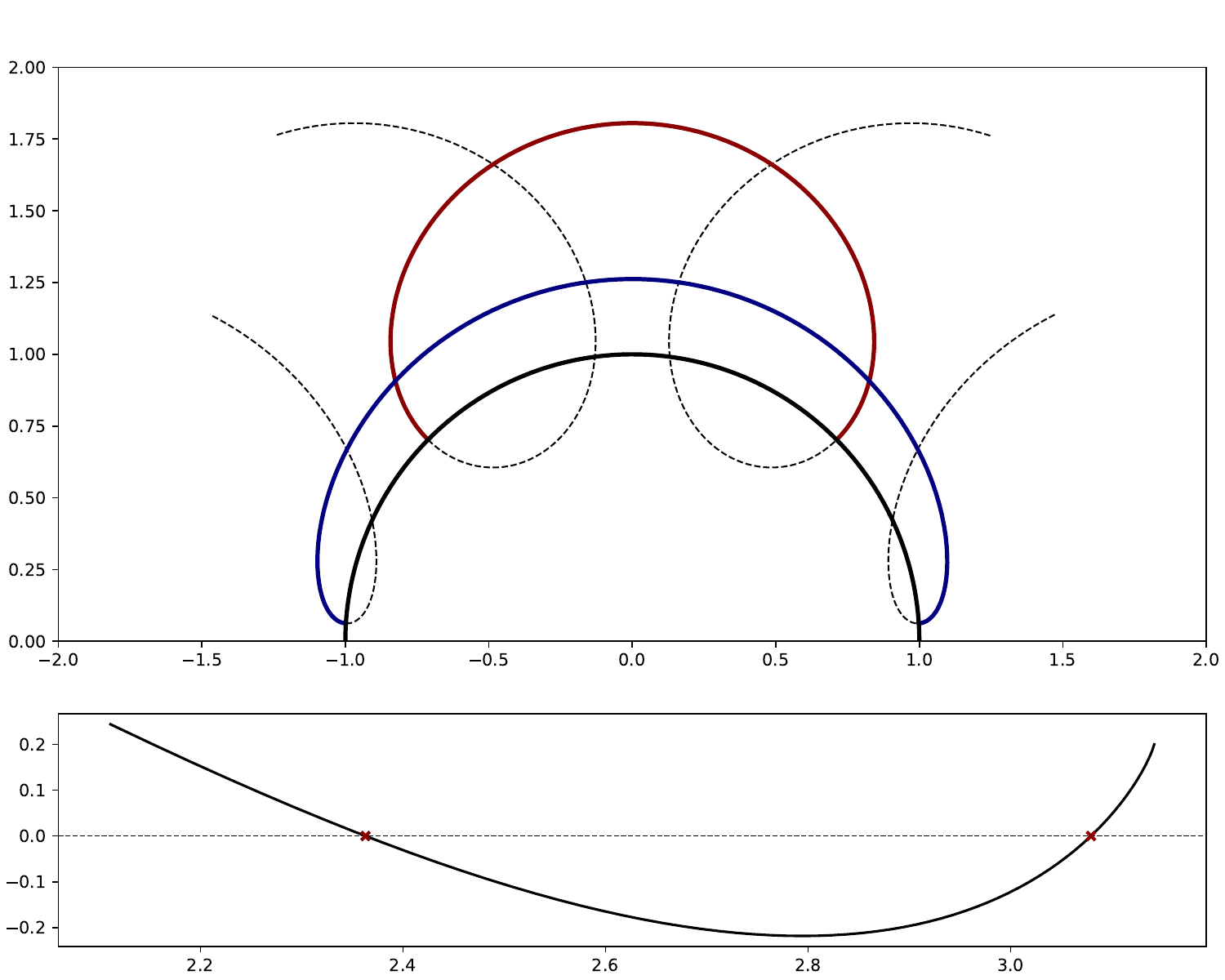}
		\caption{$N^-(a,b_{1/2})$ with $a=0.6$, $b_1\approx1.0$ (red), $b_2\approx0.3$ (blue)}
		\label{fig:fb_nodoid:branch_1_and_2}
	\end{subfigure}
	\hfill 
	\begin{subfigure}[b]{0.45\textwidth}
		\centering
		\includegraphics[width=\textwidth]{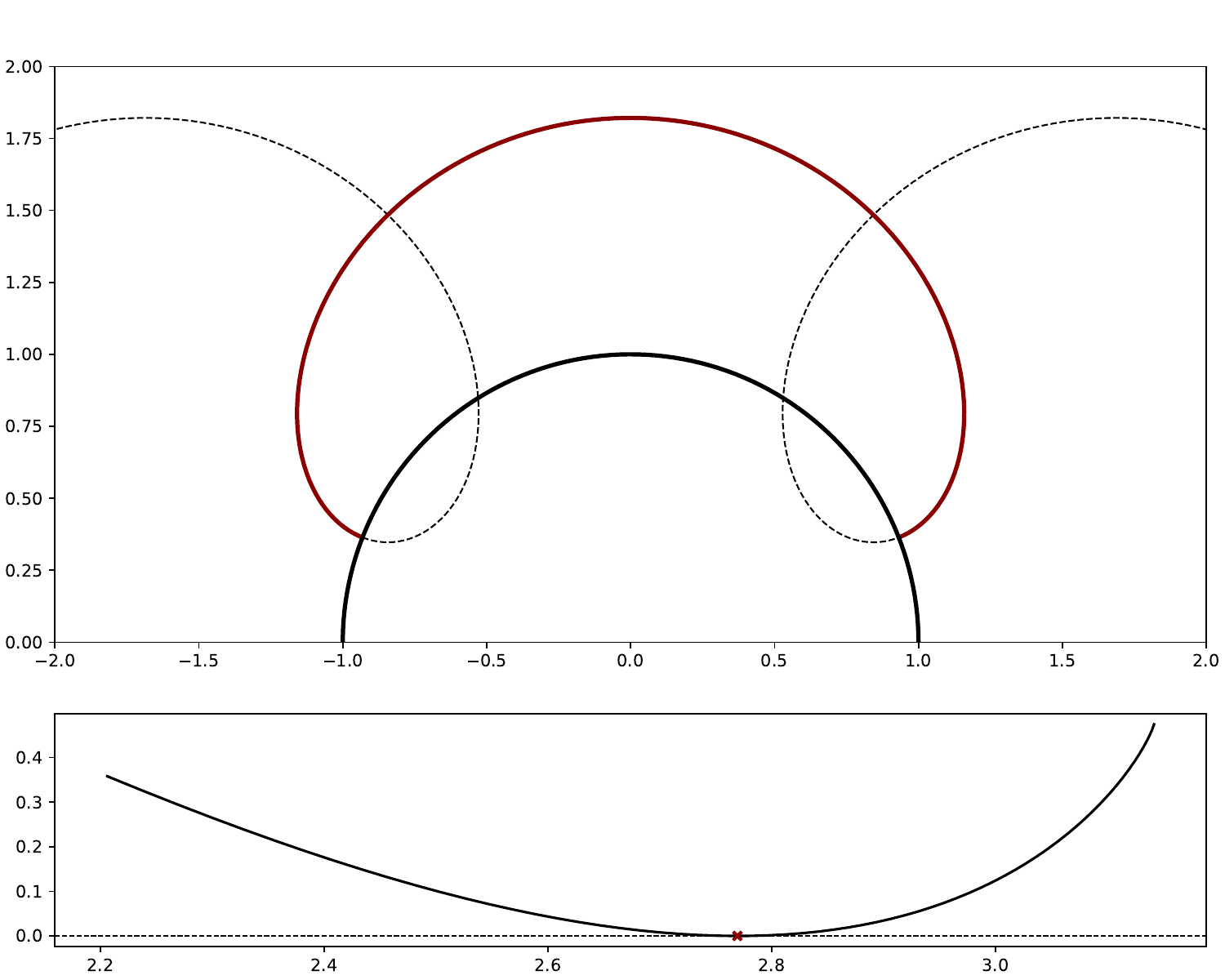}
		\caption{$N^-(a,b)$ with $a\approx0.7$, $b\approx0.8$}
		\label{fig:fb_nodoid:branch_1_and_2_border}
	\end{subfigure}
	\caption{Nodaries $N^-(a,b)$ and the upper unit circle (top) $G^-(a,\cdot)$ on $(\frac{\pi}{2}+\arctan(a),\pi)$ (bottom)}
\end{figure}

\subsubsection{Free boundary Nodoids inside of the unit ball} 
We denote by $N^+(a,b)$ with $a,b>0$ the $2$-parameter family of profile curves whose corresponding surfaces of revolution have constant mean curvature $\lambda=-\frac{1}{2a}$ where the local graph representation~\eqref{eq:intro:rho} attains a minimum at $\rho(0)=-a+\sqrt{a^2+b^2}$, see Equation~\eqref{eq:N}. In Theorem~\ref{thm:nodoid:plus} it is shown that for all $a>0$, there exist unique $b,t>0$ such that $N^+(a,b)$ restricted to $(-t,t)$ lies inside of the unit disk and meets the unit circle orthogonally at $t$, see  also \cite[Proposition~8.3]{MR3708014} and \cite[Section~2.1]{MR4979235}. This family emanates from the north pole and converges to the catenoid, see Figure~\ref{fig:fb_catenoid}. 
\begin{figure}
	\includegraphics[width=0.8\textwidth]{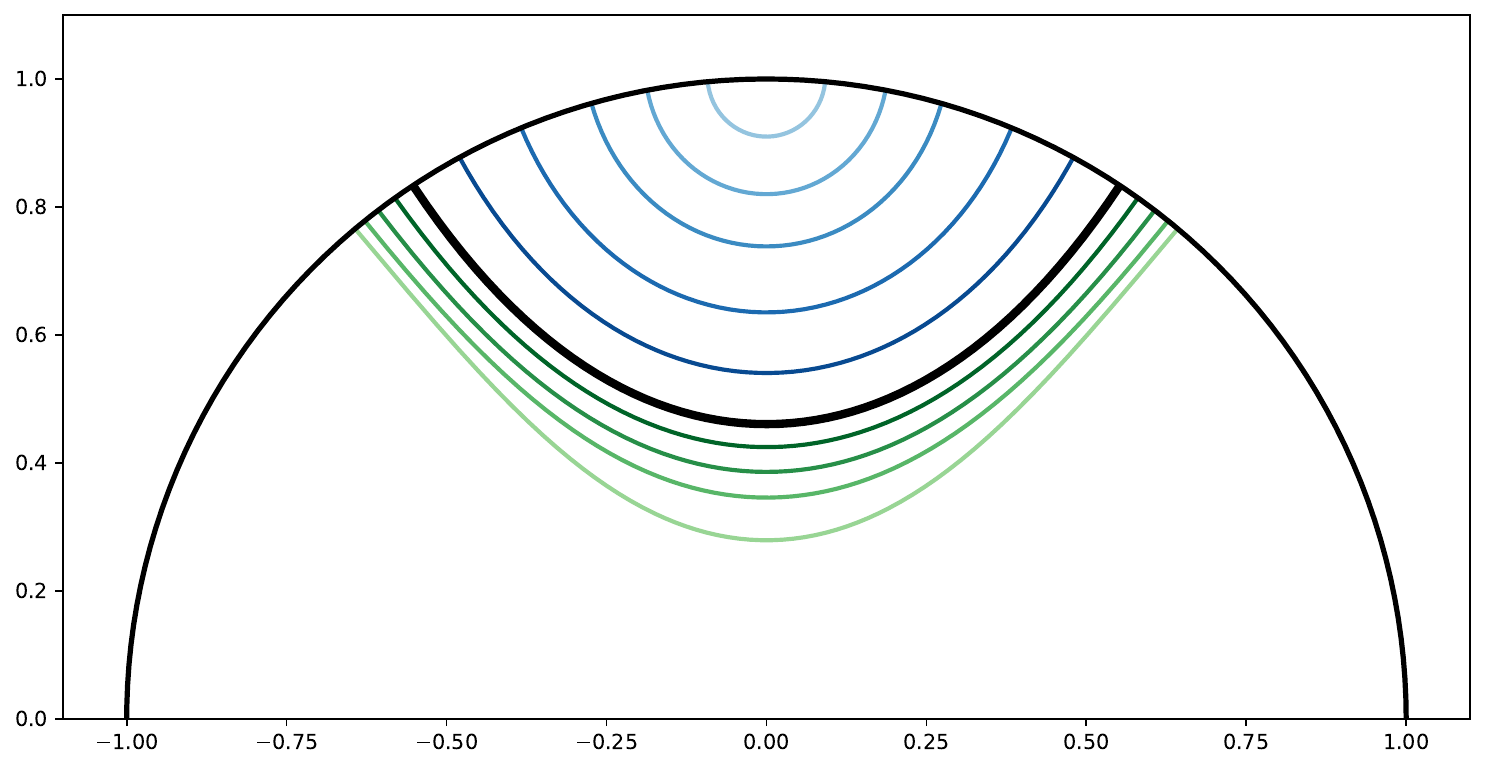}
	\caption{Free boundary nodoids (blue), the free boundary catenoid~(black), and free boundary unduloids (green)}
	\label{fig:fb_catenoid}
\end{figure}

\subsubsection{Free boundary Unduloids inside of the unit ball}
We denote by $U^\pm(a,k)$ with $a>0$ and $k\in(0,1)$ the two $2$-parameter families of profile curves whose corresponding surfaces of revolution have constant mean curvature $\frac{1}{2a}$ where the local graph representation~\eqref{eq:intro:rho} attains a minimum/maximum at $\rho(0)=a(1\mp k)$, see Equation~\eqref{eq:U}. Those unduloids with free boundary on the unit sphere can be characterized as the zeros of the functions $F^\pm$ defined in Proposition~\ref{prp:undulary:intersection}. Given $k,t>0$, there exists a unique $a>0$ such that $U^\pm(a,k)$ restricted to $(-t,t)$ lies inside of the unit disk and meets the unit circle orthogonally at $t$ if and only if $F^\pm(k,t)=0$. There exist positive functions $H^\pm$ such that $\partial_tF^\pm(\cdot,t)=\cos(t)H^\pm(\cdot,t)$, see Lemma~\ref{lem:unduloid}. Thus, $F^\pm$ behave like the sine function. In particular, there are infinitely many branches of free boundary unduloids inside the unit ball classified in Theorem~\ref{thm:unuloid}. The first branch corresponds to $F^+$ and emanates from the catenoid, see Figure~\ref{fig:fb_catenoid}. The second branch corresponds to $F^-$ and emanates from the upper unit circle, see top left in Figure~\ref{fig:fb_unduloids}. The third branch corresponds to $F^+$ and emanates from two touching upper semi circles of radius $r=\frac12$, see top right in Figure~\ref{fig:fb_unduloids}. The fourth branch corresponds to $F^-$ and emanates from three touching semi circles of radius $r=\frac13$, see bottom left in Figure~\ref{fig:fb_unduloids}. This goes on and on: for any $n$ touching semi circles of radius $r=\frac1n$, there exists a corresponding branch of free boundary unduloids, see bottom right in Figure~\ref{fig:fb_unduloids} for $n=5$. Moreover, there exists a second set of branches where the first branch emanates from the double disk, and the touching upper semi circles are shifted, see Figure \ref{fig:fb_unduloids_2}. Despite the vast number of branches, there exists no free boundary unduloid with mean curvature $H=1$, see Theorem~\ref{thm:unduloid:a}.

\begin{figure}[htbp]
	\centering
	\begin{subfigure}[b]{0.45\textwidth}
		\centering
		\includegraphics[width=\textwidth]{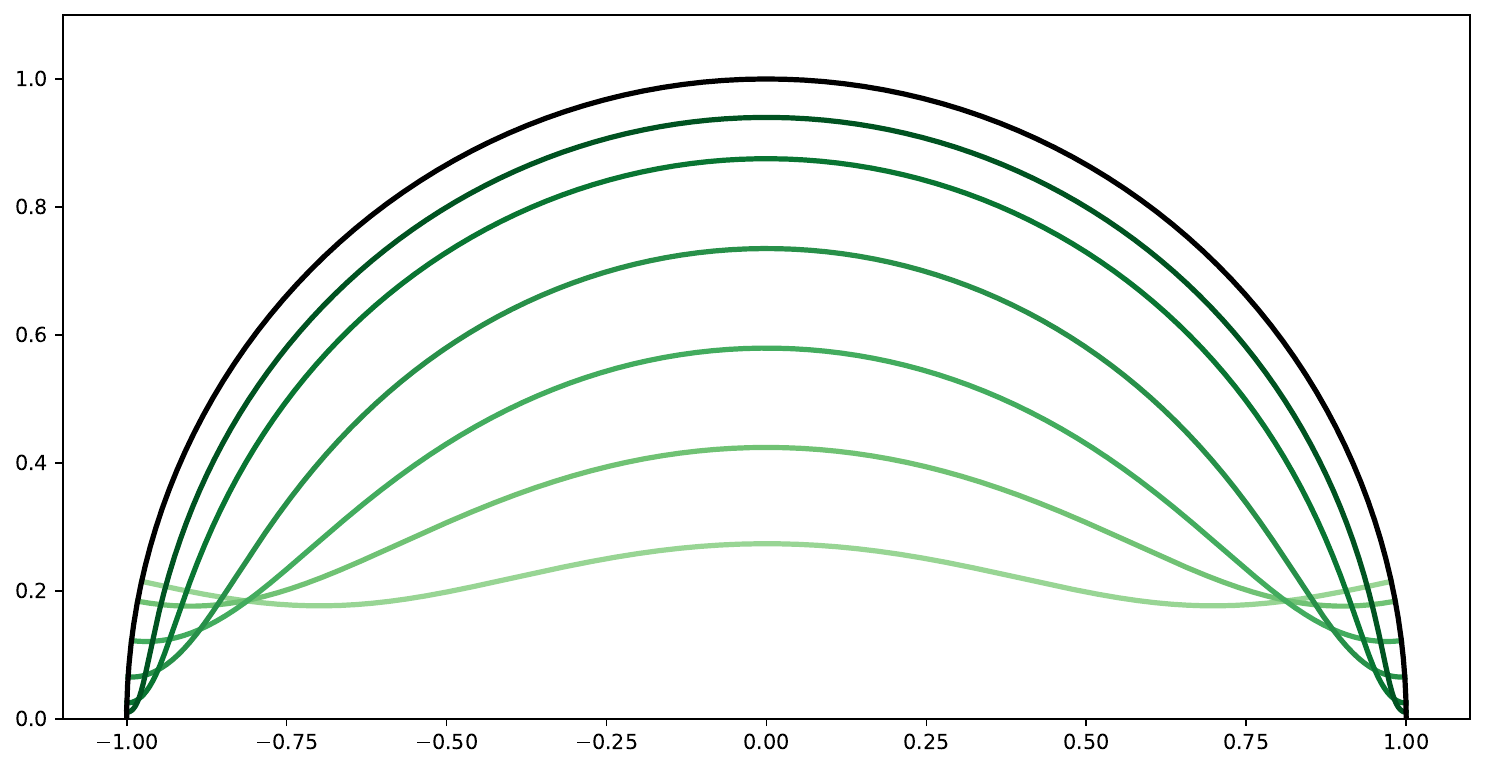}
	\end{subfigure}
	\hfill
	\begin{subfigure}[b]{0.45\textwidth}
		\centering
		\includegraphics[width=\textwidth]{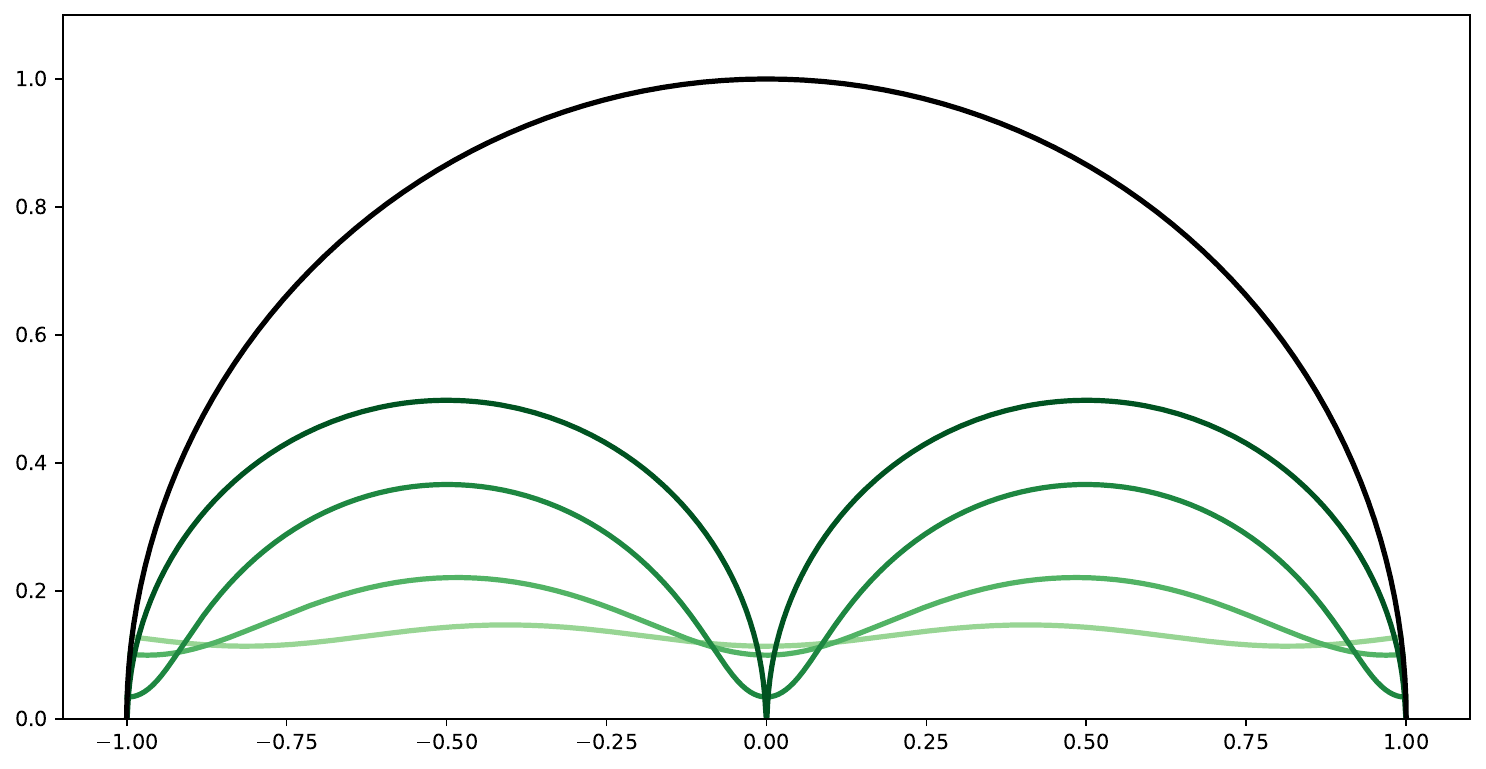}
	\end{subfigure}
	\vspace{0.5cm}
	\begin{subfigure}[b]{0.45\textwidth}
		\centering
		\includegraphics[width=\textwidth]{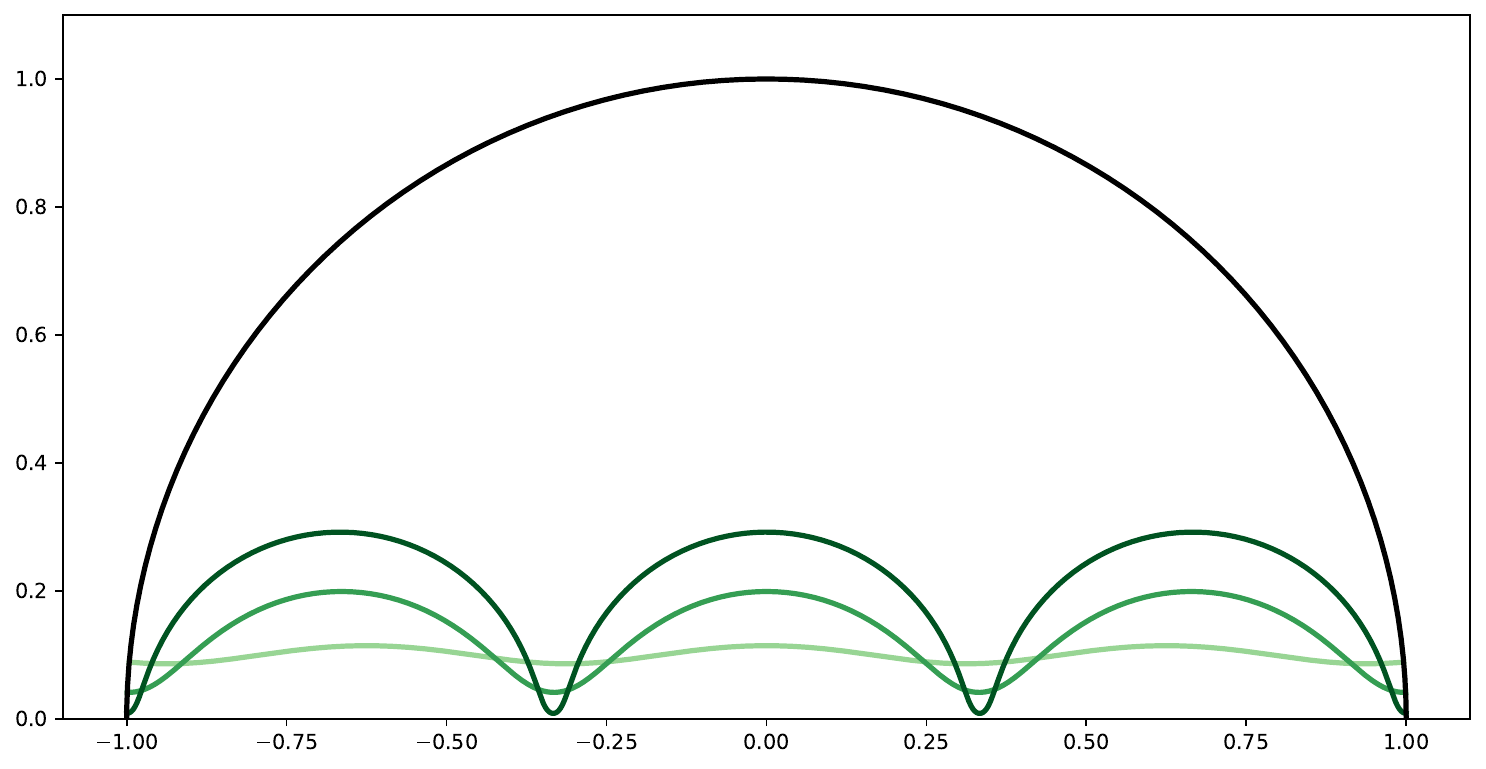}
	\end{subfigure}
	\hfill 
	\begin{subfigure}[b]{0.45\textwidth}
		\centering
		\includegraphics[width=\textwidth]{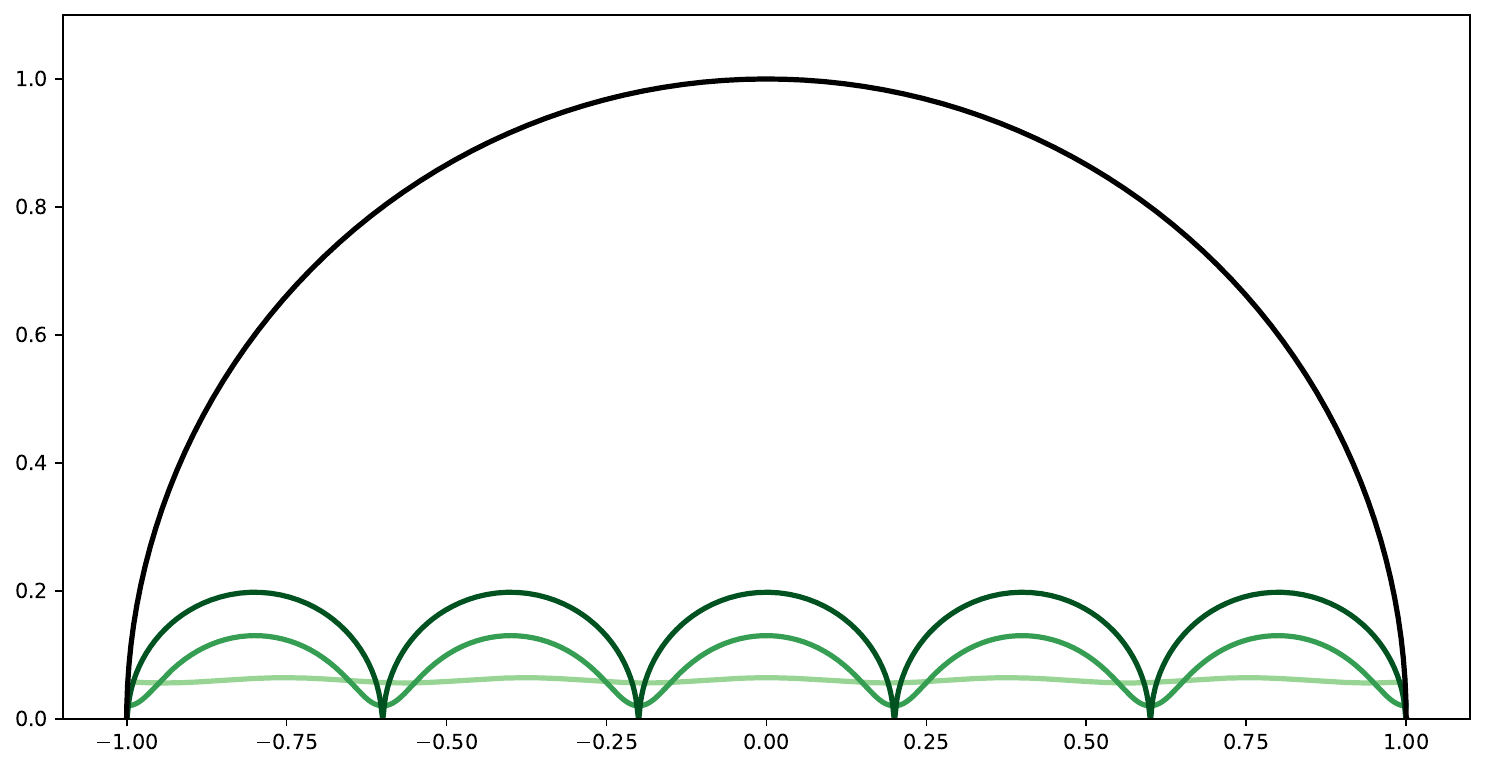}
	\end{subfigure}
	\caption{Free boundary unduloids}
	\label{fig:fb_unduloids}
\end{figure}

\begin{figure}[htbp]
	\centering
	\begin{subfigure}[b]{0.45\textwidth}
		\centering
		\includegraphics[width=\textwidth]{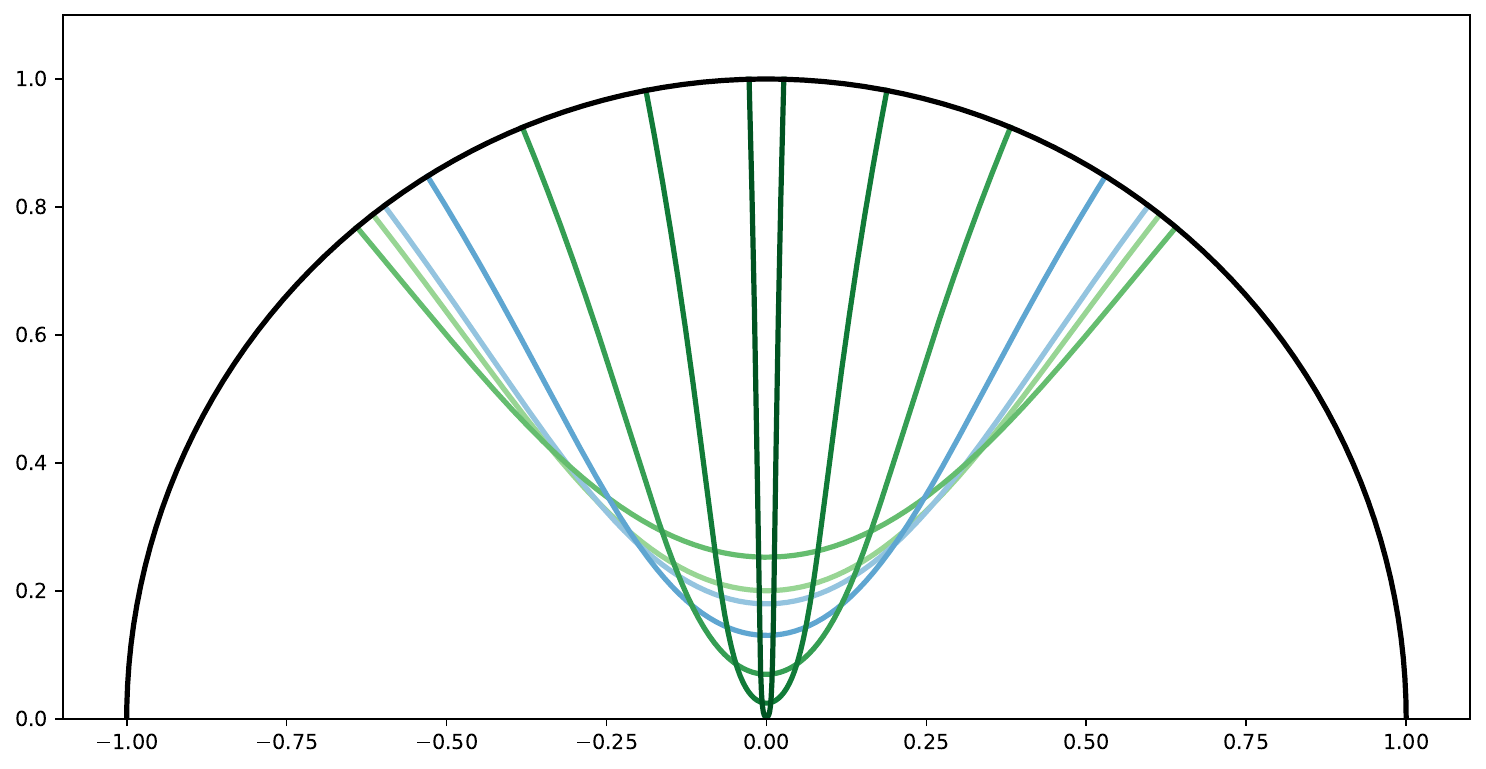}
	\end{subfigure}
	\hfill
	\begin{subfigure}[b]{0.45\textwidth}
		\centering
		\includegraphics[width=\textwidth]{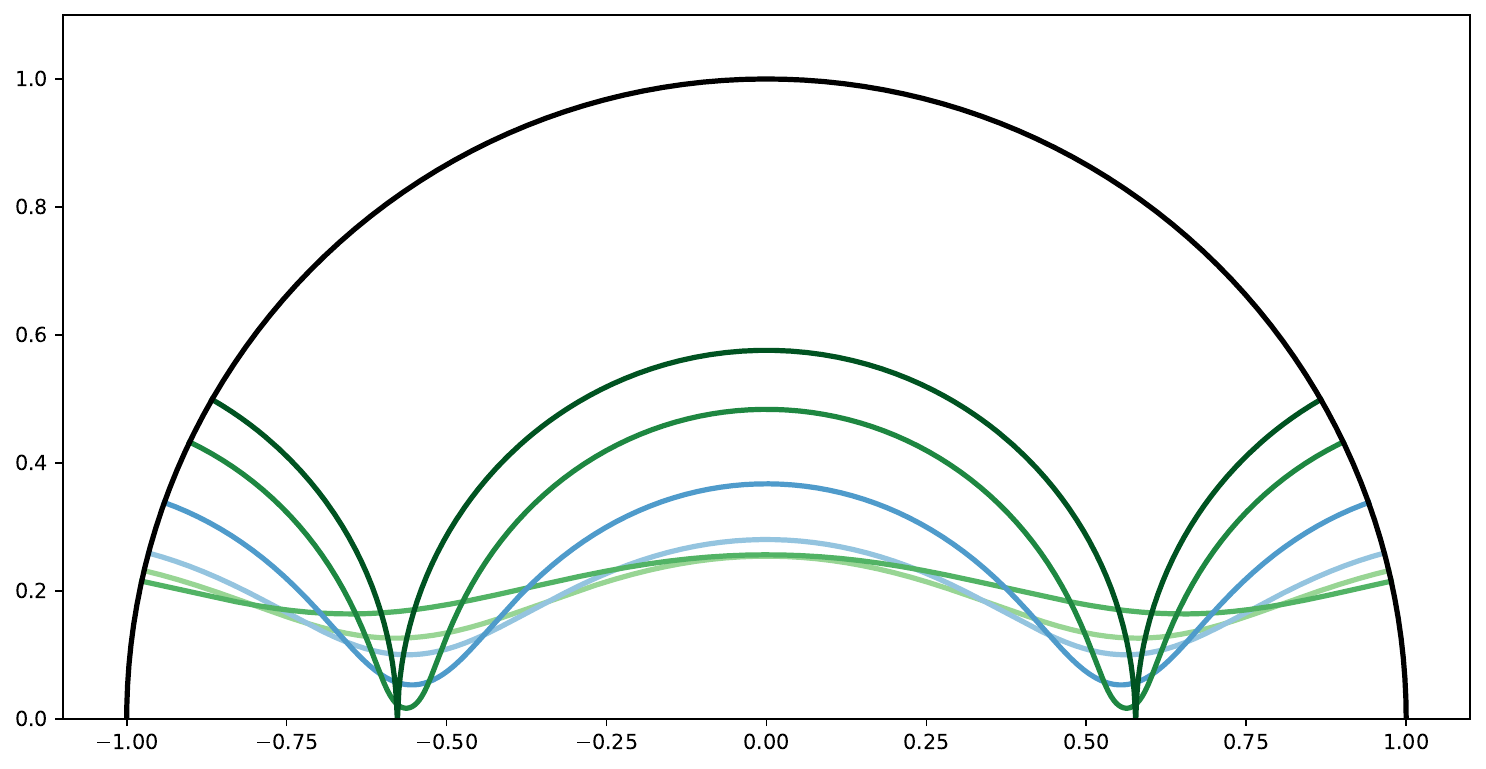}
	\end{subfigure}
	\vspace{0.5cm}
	\begin{subfigure}[b]{0.45\textwidth}
		\centering
		\includegraphics[width=\textwidth]{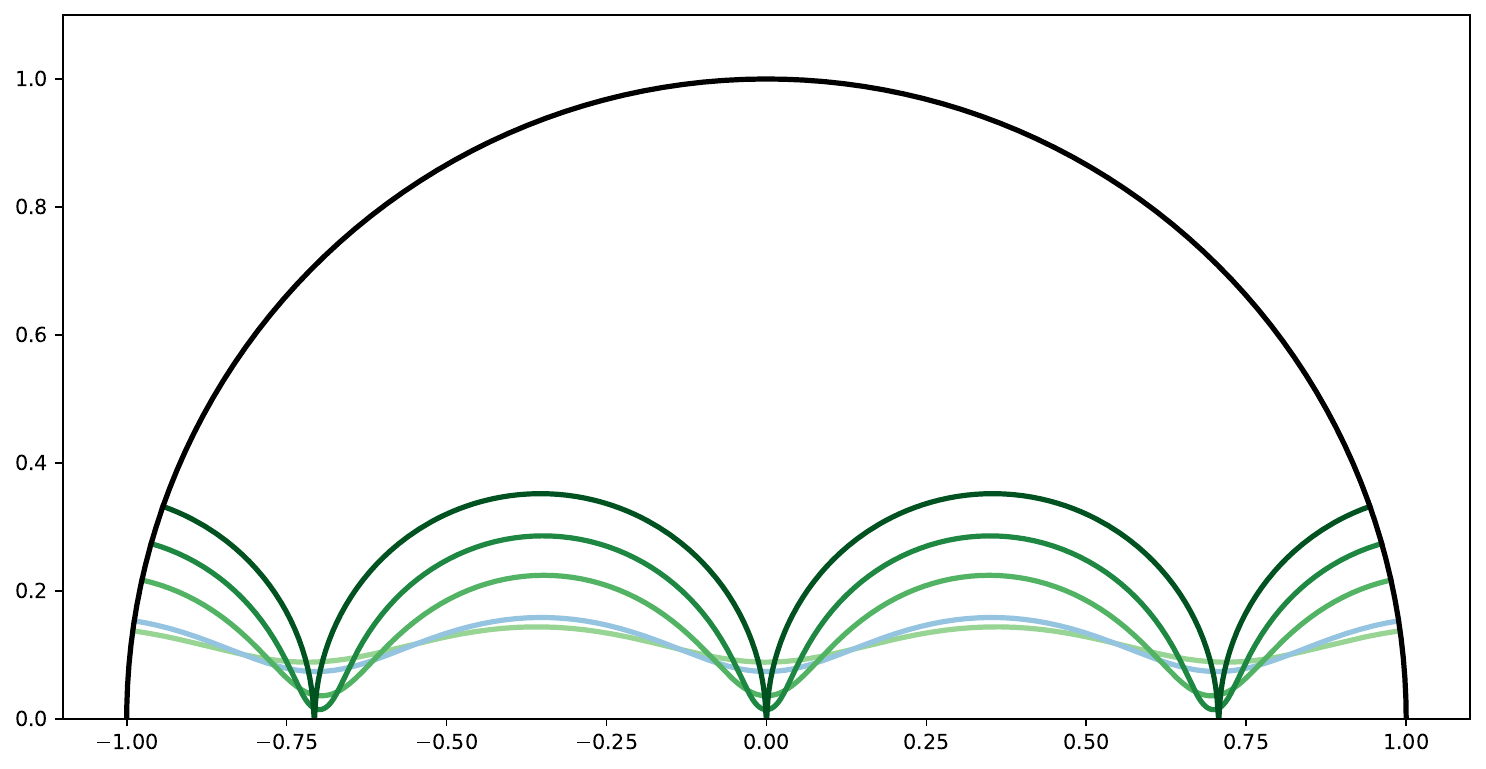}
	\end{subfigure}
	\hfill 
	\begin{subfigure}[b]{0.45\textwidth}
		\centering
		\includegraphics[width=\textwidth]{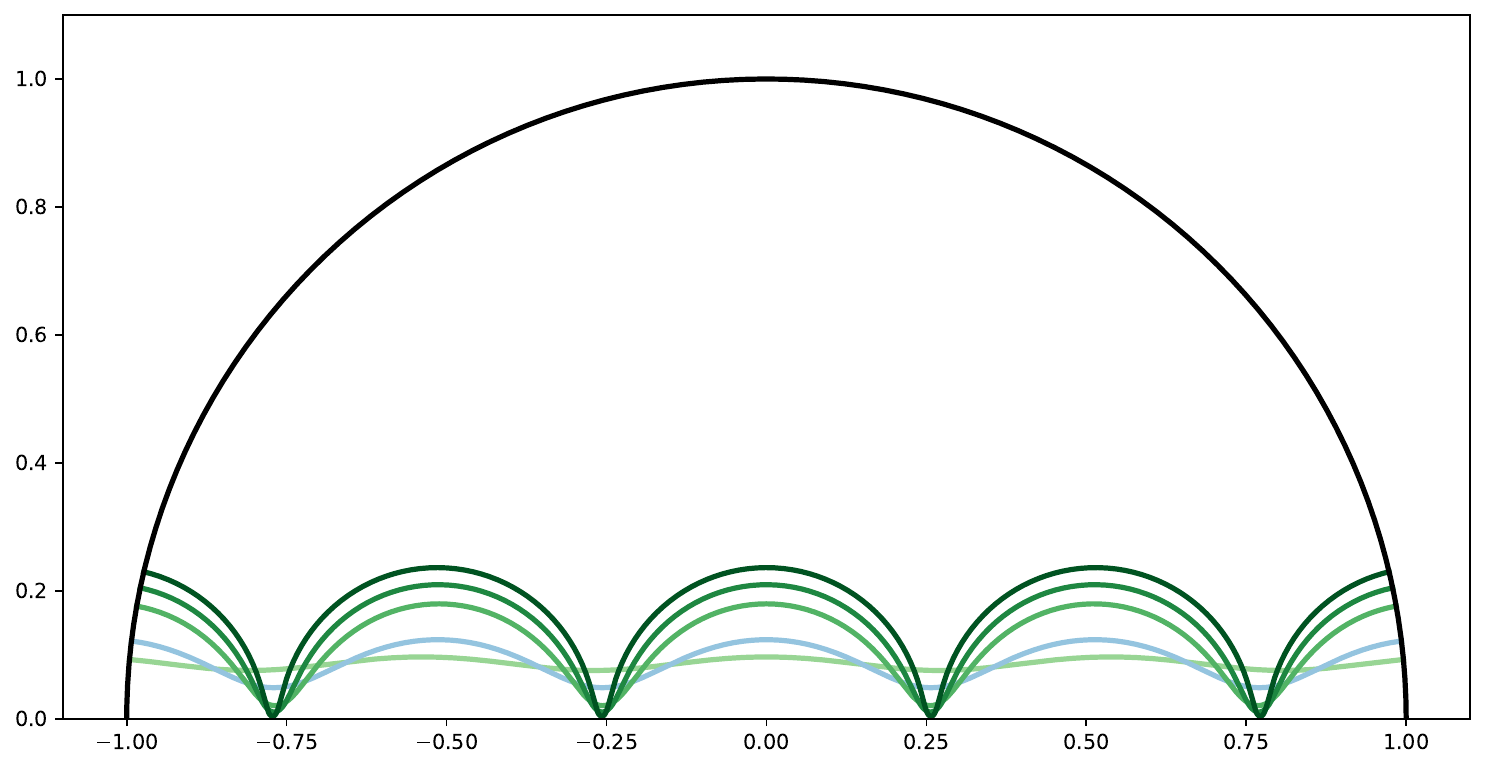}
	\end{subfigure}
	\caption{Free boundary unduloids}
	\label{fig:fb_unduloids_2}
\end{figure}

%% file: Contents/roulettes.tex
\section{Roulettes of the conics}\label{sec:roulettes}
Consider a plane curve $\alpha$ and a point $F$ attached to the curve (not necessarily on the curve). The trace $\gamma=(g,f)$ of $F$ that results as the curve $\alpha$ rolls along a straight line without slipping is called the \textit{roulette of $\alpha$ at $F$}. The main reference for this section is the work of \textsc{Bendito--Bowick--Medina}~\cite{BenditoBowickMedina2014JGSP}.

\subsection{Roulette of the parabola}
The $1$-parameter family of parabolas in the plane corresponding to $b>0$ is given by the parametrized curves
\begin{equation*}
	\alpha_\mathrm{p}:\R\to\R^2,\quad \alpha_\mathrm{p}(t)=(b\sinh^2(t),2b\sinh (t))
\end{equation*}
which describe the solutions to the equation
\begin{equation*}
	x - \frac{y^2}{4b} = 0.
\end{equation*}

\begin{proposition}[Catenary, \protect{\cite{BenditoBowickMedina2014JGSP}}]\label{prp:catenary}
	The trace of the focus $F=(b,0)$ of the parabola with parameter $b>0$ rolled vertically is
	\begin{equation*}
		(g_\mathrm{c}(t),f_\mathrm{c}(t))=(bt,b\cosh (t))
	\end{equation*}
	for $t\in\R$.
\end{proposition}

\subsection{Roulette of the ellipse}
The $2$-parameter family of ellipses in the plane corresponding to $a > b > 0$ is given by the parametrized curves
\begin{equation*}
	\alpha_\mathrm{e}:[0,2\pi)\to\R^2,\quad \alpha_\mathrm{e}(t)=(a\cos(t), b\sin(t))
\end{equation*}
which describe the solutions to the equation
\begin{equation*}
	\frac{x^2}{a^2} + \frac{y^2}{b^2} = 1.
\end{equation*}

\begin{proposition}[\protect{Undulary, \cite{BenditoBowickMedina2014JGSP}}] \label{prp:undulary}
	The traces of the foci $F_\pm=(\pm c,0)$ of the ellipse with parameters $a>b>0$ and $c=\sqrt{a^2-b^2}$ rolled horizontally are
	\begin{align}\label{eq:undulary:g}
		g_{\mathrm u}^{\pm}(t) & = \int_{0}^{t} \sqrt{a^2-c^2\cos^2(x)} \,\d x \mp \frac{c\sin(t)(a\mp c\cos(t))}{\sqrt{a^2-c^2\cos^2(t)}}, \\ \label{eq:undulary:f}
		f_{\mathrm u}^{\pm}(t) & = \frac{b(a\mp c\cos(t))}{\sqrt{a^2-c^2\cos^2(t)}}
	\end{align}
	for $t \in \R$.
\end{proposition}

\begin{remark}\label{rem:undulary}
	The two branches $(g_\mathrm{u}^+,f_\mathrm{u}^+)$ and $(g_\mathrm{u}^-,f_\mathrm{u}^-)$ corresponding to the foci $(c,0)$ and $(-c,0)$ are the two horizontal translations of the same curve referred to as \emph{undulary} that are symmetric with respect to the $y$-axis. We have 
	\begin{align*}
		f_\mr{u}^\pm(-t)&=f_\mr{u}^\pm(t),\\
		g_\mr{u}^\pm(-t)&=-g_\mr{u}^\pm(t)
	\end{align*} 
	and
	\begin{align*}
		f_\mr{u}^+(t+\pi)&=f_\mr{u}^-(t),\\
		g_\mr{u}^+(t+\pi)&=g_\mr{u}^-(t)+\int_{0}^{\pi}\sqrt{a^2-c^2\cos^2(x)} \,\d x.
	\end{align*} 
\end{remark}

\begin{figure}
	\centering
	\includegraphics[width=1\textwidth]{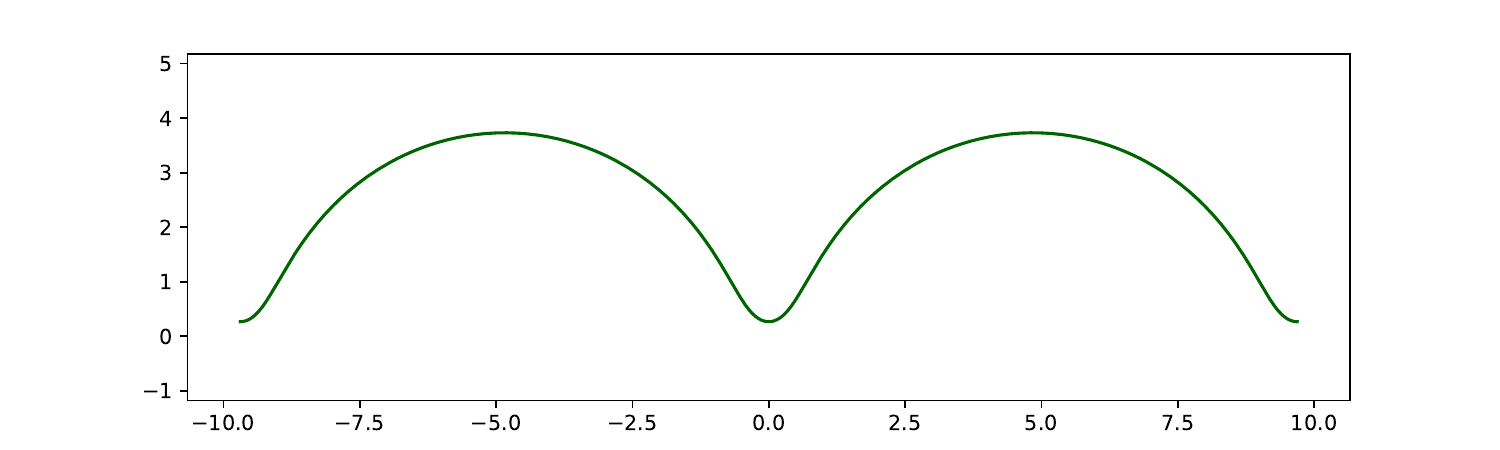}
	\caption{Undulary $(f_\mathrm{u}^+,g_\mathrm{u}^+)$ with $a=2$, $b=1$, and $t\in[-2\pi,2\pi]$}
	\label{fig:undulary}
\end{figure}

\subsection{Roulette of the hyperbola}
The $2$-parameter family of hyperbola branches in the half-plane on the right corresponding to $a,b>0$ is given by the parametrized curves 
\begin{equation*}
	\alpha_\mathrm{h}:\R\to\R^2,\quad \alpha_\mathrm{h}(t)=(a\cosh(t),b\sinh(t))
\end{equation*}
which describe the solutions to the equation
\begin{equation*}
	\frac{x^2}{a^2}-\frac{y^2}{b^2}=1
\end{equation*}
with $x>0$.

\begin{proposition}[\protect{Nodary, \cite{BenditoBowickMedina2014JGSP}}]
	The traces of the foci $F_\pm=(\pm c,0)$ of the hyperbola branch on the right with parameters $a,b>0$ and $c=\sqrt{a^2+b^2}$ rolled vertically are
	\begin{align*}
		g_{\mathrm n}^{\pm}(t) &= \int_{0}^{t} \sqrt{c^2\cosh^2(x)-a^2} \,\d x - \frac{c\sinh(t)(c\cosh(t)\mp a)}{\sqrt{c^2\cosh^2(t)-a^2}},\\
		f_{\mathrm n}^{\pm}(t) &= \frac{b(c\cosh(t)\mp a)}{\sqrt{c^2\cosh^2(t)-a^2}}
	\end{align*}
	for $t\in\R$.
\end{proposition}

\begin{remark}\label{rem:nodary}
	After performing the parameter transformation $t = \arcsinh u$ and $u=\frac{b}{c}\tan s$, the two branches $(g_\mathrm{n}^+,f_\mathrm{n}^+)$ and $(g_\mathrm{n}^-,f_\mathrm{n}^-)$ corresponding to the foci $(c,0)$ and $(-c,0)$ can be glued together to a single, smooth periodic curve referred to as \emph{nodary}. The resulting two horizontal translations that are symmetric with respect to the vertical axis are 
	\begin{align}\label{eq:nodary:g}
		\displaystyle g_{\mathrm n}^\pm(s) &= -a^2b^2 \int_0^s \frac{\sin^2(x)}{{\left(b^2+a^2\cos^2(x)\right)}^{\frac{3}{2}}} \,\d x -\frac{a^2\sin(s)\cos(s)}{\sqrt{b^2+a^2\cos^2(s)}} \pm a\sin(s),\\ \label{eq:nodary:f}
		\displaystyle f_{\mathrm n}^\pm(s) &= \sqrt{b^2+a^2\cos^2(s)} \mp a\cos(s)
	\end{align}
	for $s\in\R$, see \cite[Section 5.1]{Scharrer22NonAna}. We have
	\begin{align*}
		f_\mr{n}^\pm(-s)&=f_\mr{n}^\pm(s),\\
		g_\mr{n}^\pm(-s)&=-g_\mr{n}^\pm(s)
	\end{align*} 
	and
	\begin{align*}
		f_\mr{n}^+(s+\pi)&=f_\mr{n}^-(s),\\
		g_\mr{n}^+(s+\pi)&=g_\mr{n}^-(s)-a^2b^2 \int_0^\pi \frac{\sin^2(x)}{{\left(b^2+a^2\cos^2(x)\right)}^{\frac{3}{2}}} \,\d x.
	\end{align*} 
\end{remark}

\begin{figure}
	\centering
	\includegraphics[width=0.7\textwidth]{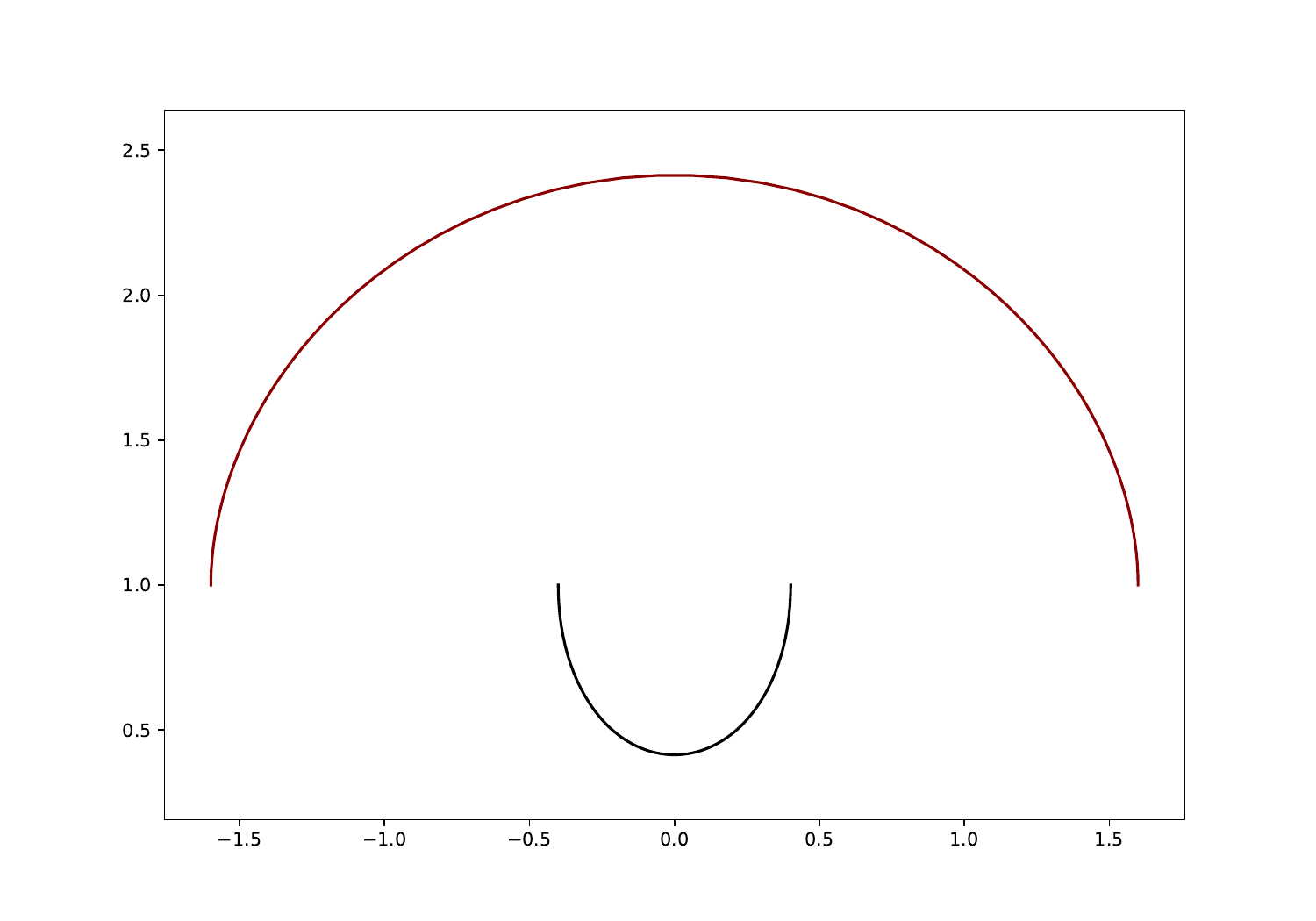}
	\caption{Roulette of the hyperbola at $F_+$ (black) and $F_-$ (red) with a=1, b=1}\label{fig:Nodary1}
\end{figure}

\begin{figure}
    \centering
    \includegraphics[width=1\textwidth]{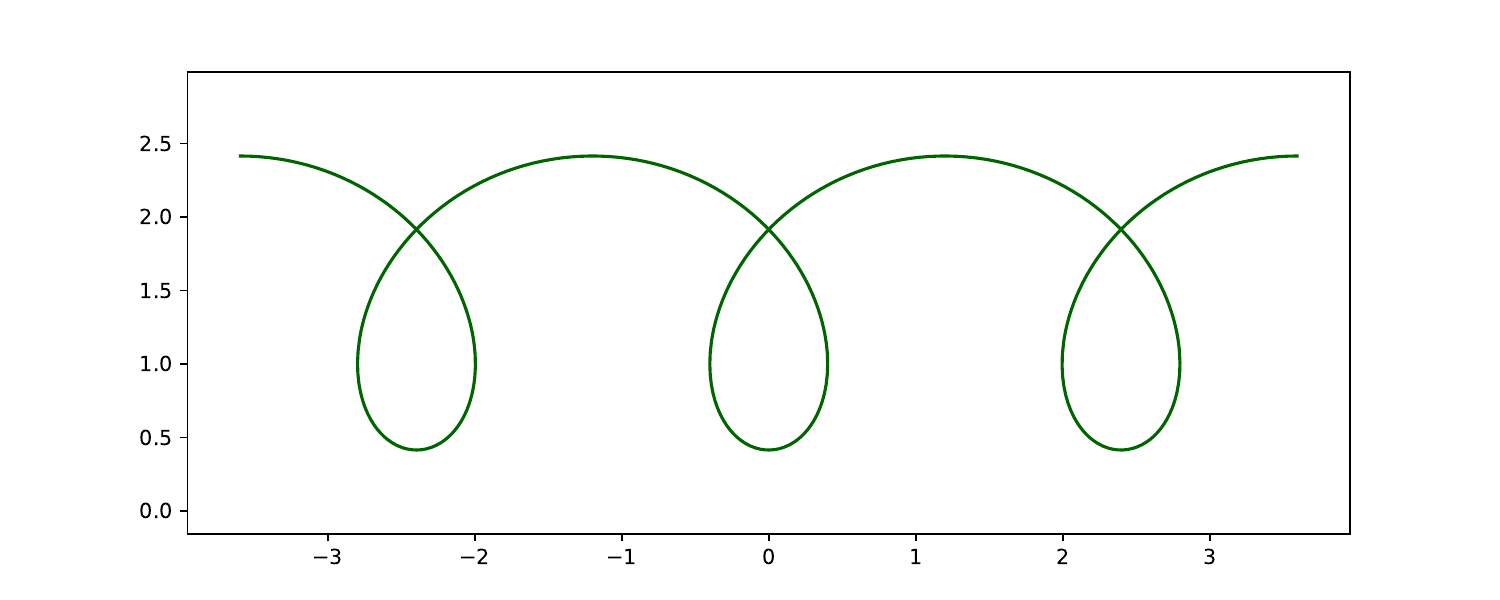}
    \caption{Nodary $(g_\mr{n}^+,f_\mr{n}^+)$ with $a=1$, $b=1$, and $s\in [-3\pi,3\pi]$} 
    \label{fig:Nodary2}
\end{figure}

\subsection{Derivatives of the roulettes}

\begin{proposition}\label{prp:undulary:derivatives}
	Let $a>b>0$, $c=\sqrt{a^2-b^2}$, and 
	\begin{equation}\label{eq:prp:undulary:derivatives}
		h^\pm(t)=\displaystyle\frac{ab}{(a\pm c \cos(t))\sqrt{a^2-c^2\cos^2(t)}}
	\end{equation}
	for $t\in \R$. Then, the components of the undulary given in \eqref{eq:undulary:g} and \eqref{eq:undulary:f} satisfy
	\begin{enumerate}\upshape
		\item \label{it:prp:undulary:derivatives:g} $\displaystyle {g^\pm_\mr{u}}'(t)= b h^\pm(t)$; 
		\item \label{it:prp:undulary:derivatives:f} $\displaystyle {f^\pm_\mr{u}}'(t)= \pm c h^\pm(t)\sin(t)$;
		\item \label{it:prp:undulary:derivatives:sign-g_pm} $g_\mr{u}^\pm(t)>0$ for $t>0$.
	\end{enumerate} 
\end{proposition}

\begin{proof}
	Statements \eqref{it:prp:undulary:derivatives:g} and \eqref{it:prp:undulary:derivatives:f} are straight forward computations. Since $g^\pm(0)=0$ and $h^\pm>0$, we conclude that \eqref{it:prp:undulary:derivatives:g} implies \eqref{it:prp:undulary:derivatives:sign-g_pm}.
\end{proof}

\begin{proposition}\label{prp:nodary:derivatives}
    Let $a,b>0$, $c=\sqrt{a^2+b^2}$, and 
    \begin{equation*}
    	Q(s) = \sqrt{b^2 + a^2\cos^2(s)}, \quad \Lambda^\pm(s)=\frac{a \cos(s)}{Q(s)}\mp1	
    \end{equation*}
    for $s\in\R$. 
    Then,
    \begin{equation}\label{eq:prp:nodary:derivatives:sign-Lambda}
    	\Lambda^+<0,\qquad \Lambda^->0,
    \end{equation} 
    and the components of the nodary given in \eqref{eq:nodary:g} and \eqref{eq:nodary:f} satisfy
    \begin{enumerate}\upshape
        \item \label{it:prp:nodary:derivatives} $\displaystyle g_\mathrm{n}^\pm(s) = -a^2\int_0^s \frac{\cos^2(x)}{Q(x)} \,\d x \pm a\sin(s)$; 
        \item \label{it:prp:nodary:derivatives:g} $\displaystyle {g_\mathrm{n}^\pm}'(s) = -a\cos(s)\Lambda^\pm(s)$;
        \item \label{it:prp:nodary:derivatives:f} $\displaystyle {f_\mathrm{n}^\pm}'(s) = -a\sin(s)\Lambda^\pm(s)$.
    \end{enumerate}
\end{proposition}

\begin{proof}
    One readily verifies
    \begin{align*}
        \frac{\d}{\d s} \frac{a^2\sin(s)\cos(s)}{Q(s)} = - \frac{a^2b^2\sin^2(s)}{{Q(s)}^3} + \frac{a^2\cos^2(s)}{{Q(s)}}
    \end{align*}
    which, by the fundamental theorem of calculus, implies \eqref{it:prp:nodary:derivatives}-\eqref{it:prp:nodary:derivatives:f}. 

\end{proof}

%% file: Contents/Delaunay.tex
\section{Delaunay surfaces}
In 1841, the French astronomer and mathematician \textsc{Delaunay}~\cite{Delaunay1841JMPA} studied surfaces of revolution with constant mean curvature. He showed that apart from the plane and the sphere, the generating curves of these surfaces are the roulettes of the foci of the non-degenerate conic sections. The surfaces created by revolving the roulette of the foci of the circle, parabola, ellipse, and hyperbola are the cylinder, catenoid, unduloid, and nodoid, respectively. These surfaces are known as \emph{Delaunay surfaces}. 

Since then, various approaches to parametrizing Delaunay surfaces have been introduced. A list of some of the known parametrizations can be found in the appendix of~\cite{MladenovaMladenov24Mathematics}.

\subsection{Surfaces of revolution}
Given two real valued functions $f,g$ defined on an interval $(t_1,t_2)$ with $-\infty\le t_1<t_2\le \infty$ and $f>0$, the surface of revolution corresponding to the profile curve $\gamma = (g,f)$ is parametrized by
\begin{equation*}
	\boldsymbol x:(t_1,t_2)\times[0,2\pi)\to \mathbb R^3,\quad \boldsymbol x(t,v) = \bigl(f(t)\cos v,f(t)\sin v,g(t)\bigr).
\end{equation*}
The coefficients of its first and second fundamental form are formally given by
\begin{align*}
	E &= {(f')}^2 + {(g')}^2, & F &= 0, & G &= f^2\\
	L &= \|\gamma'\|^{-1} \left(f'g''-f''g'\right), & M &= 0, & N &= \|\gamma'\|^{-1}fg'.
\end{align*}
The corresponding unit normal and resulting mean curvature are
\begin{equation*}
	\boldsymbol n(\cdot,v) = \|\gamma'\|^{-1}(-g'\cos v,-g'\sin v, f'),\quad  H = \frac{1}{2} \Bigl( \frac{f'g''-f''g'}{{\|\gamma'\|}^{3}} + \frac{g'}{f \|\gamma'\|} \Bigr).
\end{equation*}

\subsection{Catenoids}
A surface of revolution has constant vanishing mean curvature $H=0$ if and only if it is either contained in a plane, or its generating curve is a piece of a catenary $(g_\mathrm{c},f_\mathrm{c})$ as defined in Proposition \ref{prp:catenary} for some parameter $b>0$.

\subsection{Unduloids}
A surface of revolution has constant positive mean curvature if and only if it is either contained in a cylinder, or it is contained in a round sphere, or its generating curve is a piece of an undulary $(g_\mathrm{u},f_\mathrm{u})$ as defined in Remark~\ref{rem:undulary} for some parameters $a>b>0$ with $H=\frac{1}{2a}$.

\subsection{Nodoids}\label{sec:Delaunay:nodoid}
A surface of revolution has constant negative mean curvature if and only if its generating curve is a piece of a nodary $(g_\mathrm{n},f_\mathrm{n})$ as defined in Remark~\ref{rem:nodary} for some parameters $a,b>0$ with $H=-\frac{1}{2a}$.

%% file: Contents/intersections.tex
\section{Orthogonal intersections of Delaunay surfaces with the unit sphere}

In this section, we will analyze orthogonal intersections of Delaunay surfaces with the unit sphere. Since Delaunay surfaces just like the unit sphere are surfaces of revolution, we can reduce the problem to analyzing intersections of the generating curves with the unit circle. That is, we aim to determine all generating curves $(g,f)$ of Delaunay surfaces that are symmetric with respect to the $y$-axis for which there exists $t>0$ in the domain of $(g,f)$ solving the system
\begin{align}\label{it:circle} \tag{I}
	1&={g(t)}^2 + {f(t)}^2, \\ \label{it:tangent} \tag{II}
	0&=g'(t)f(t) - g(t)f'(t).
\end{align}
The first equation means that $(g,f)$ meets the unit circle at $t$, while the second guarantees orthogonality at the point of intersection.

\subsection{Free boundary catenoid} \label{sec:fb_catenoid}
For the catenary 
\begin{equation*}
	(g_\mathrm{c}(t),f_\mathrm{c}(t))=(bt,b\cosh t)
\end{equation*}
defined in Proposition \ref{prp:catenary}, the system \eqref{it:circle}-\eqref{it:tangent} reads
\begin{equation*}
	t=\coth t,\qquad 	b=\frac{1}{\sqrt{t^2+\cosh^2t}}
\end{equation*}
which has a unique solution $(t^*,b^*)$. A numerical analysis gives $t^*\approx1.199678$ and $b^*\approx0.460485$.

\subsection{Free boundary unduloids}
Given $a>0$ and $k\in(0,1)$, we let $b=a\sqrt{1-k^2}$ and denote by 
\begin{equation}\label{eq:U}
	U^\pm(a,k)\vcentcolon=(g_\mathrm{u}^\pm,f_\mathrm{u}^\pm)
\end{equation}
the two undularies corresponding to the parameters $a,b$ that are symmetric with respect to the $y$-axis, see \eqref{eq:undulary:g}-\eqref{eq:undulary:f}. 

\begin{proposition}\label{prp:undulary:system}
	For the undulary $U^\pm(a,k)$, the system \eqref{it:circle}-\eqref{it:tangent} is equivalent to
	\begin{equation*}
		 g_\mathrm{u}^\pm(t)=\frac{\sqrt{1-k^2}}{\sqrt{1-k^2\cos^2(t)}}, \qquad 
		 f_\mathrm{u}^\pm(t)=\pm\frac{k\sin(t)}{\sqrt{1-k^2\cos^2(t)}}.
	\end{equation*}
\end{proposition}

\begin{proof}
	By the definition in \eqref{eq:prp:undulary:derivatives}, we have $h^\pm>0$. Thus, the statement follows from \eqref{it:prp:undulary:derivatives:g}-\eqref{it:prp:undulary:derivatives:sign-g_pm} of Proposition \ref{prp:undulary:derivatives}.	
\end{proof}

\begin{proposition}[Characterizing orthogonal intersections of undulary and circle]\label{prp:undulary:intersection}
	Given $0<k<1$, $t>0$, and letting
	\begin{equation*}
		I(k,t)\vcentcolon=\int_0^t\sqrt{1-k^2\cos^2(x)}\,\mr dx,
	\end{equation*}
	there exists $a>0$ such that $U^\pm(a,k)$ intersects the unit circle orthogonally at $t$ if and only if 
	\begin{equation}\label{eq:prp:undulary:intersection:sign}
		\pm\sin(t)>0
	\end{equation}
	and
	\begin{equation}\label{eq:prp:undulary:intersection:F}
		0=F^\pm(k,t) \vcentcolon= k\sin(t)I(k,t)\mp(1\mp k\cos(t))\sqrt{1-k^2\cos^2(t)}.
	\end{equation}
	In this case, $a$ is uniquely given by
	\begin{equation}\label{eq:prp:undulary:intersection:a}
		a^\pm = \pm\frac{k}{\sqrt{1-k^2}}\frac{\sin(t)}{1\mp k\cos(t)}.
	\end{equation}
\end{proposition}

\begin{proof}
	Equation \eqref{eq:prp:undulary:intersection:a} is equivalent to the second equation of Proposition~\ref{prp:undulary:system} by means of which the first equation of Proposition~\ref{prp:undulary:system} is equivalent to $F(k,t)=0$. 
\end{proof}

\begin{definition}\label{def:elliptic_integral}
	The \emph{complete elliptic integrals} of the \emph{first} and \emph{second kind} are
	\begin{equation*}
		K(k)=\int_0^\frac{\pi}{2}\frac{\mr dx}{\sqrt{1-k^2\sin^2(x)}},\qquad E(k)\vcentcolon=\int_0^\frac{\pi}{2}\sqrt{1-k^2\sin^2(x)}\mr dx
	\end{equation*}
	where $0\le k\le 1$.
\end{definition}

\begin{remark}\label{rem:def:elliptic_integral}
	The derivatives are
	\begin{equation}\label{eq:elliptic_integral:derivative}
		K'(k)=\frac{E(k)}{k(1-k^2)}-\frac{K(k)}{k},\qquad E'(k)=\frac{E(k)-K(k)}{k}
	\end{equation}	
	see e.g.\ \cite[710.00, 710.01]{ByrdFriedman1971Springer}. Moreover,
	\begin{equation*}
		K\ge \pi/2,\qquad 1\le E\le \pi/2
	\end{equation*}	
	by \cite[900.00, 900.07]{ByrdFriedman1971Springer}, and
	\begin{equation*}
		\lim_{k\nearrow1}\Bigl(K(k)-\log(4/\sqrt{1-k^2})\Bigr)=0
	\end{equation*}	
	by \cite[112.01]{ByrdFriedman1971Springer}. In particular,
	\begin{equation*}
		\lim_{k\nearrow1}K(k)=\infty
	\end{equation*}
 	and
	\begin{equation*}
		K'(k)>0,\quad E'(k)<0\qquad\text{for $0<k\le1$}.
	\end{equation*}
\end{remark}

\begin{lemma}\label{lem:unduloid}
	Let $F^\pm,I$ be as in Proposition \ref{prp:undulary:intersection}. Define
	\begin{equation}\label{eq:lem:unduloid:h}
		h^\pm_l(k)\vcentcolon=F^\pm\Bigl(k,l\pi+\frac{\pi}{2}\Bigr)=k(-1)^l(2l+1)E(k)\mp1 
	\end{equation}
	for all nonnegative integers $l$ and $0<k<1$, as well as
	\begin{equation*}
		H^\pm(k,t)\vcentcolon=kI(k,t)\mp (1\mp k\cos(t))\frac{k^2\sin(t)}{\sqrt{1 - k^2\cos^2(t)}}.
	\end{equation*}
	for all $0<k<1$ and $t>0$. Then, there holds
	\begin{align}\label{eq:lem:unduloid:h-prime}
		{h_l^\pm}'(k)&=(-1)^l(2l+1)(2E(k)-K(k)),\\ \label{eq:lem:unduloid:h-pprime}
		{h_l^\pm}''(k)&=(-1)^l(2l+1)\Bigl(\frac{2E(k)-K(k)}{k}-\frac{E(k)}{k(1-k^2)}\Bigr),
	\end{align}
	as well as
	\begin{equation}\label{eq:lem:unduloid:partialF}
		\partial_tF^\pm(k,t)=\cos(t)H^\pm(k,t),\qquad H^\pm(k,t)>0.
	\end{equation}
	Moreover, there exists a unique $\kappa_0\in[0,1]$ such that 
	\begin{equation}\label{eq:lem:unduloid:2E=K}
		2E(\kappa_0)-K(\kappa_0)=0.
	\end{equation}
	It holds $\kappa_0\in(0,1)$.
\end{lemma}

\begin{proof}
	Using \eqref{eq:elliptic_integral:derivative}, one readily verifies \eqref{eq:lem:unduloid:h-prime} and \eqref{eq:lem:unduloid:h-pprime}. Moreover, by a straight forward computation, we have $\partial_tF^\pm(k,t)=\cos(t)H^\pm(k,t)$ and
	\begin{equation*}
		\partial_tH(k,t)=\frac{k(1-k^2)}{(1\pm k\cos(t))\sqrt{1-k^2\cos^2(t)}}
	\end{equation*}
	which is strictly positive for $0<k<1$. Since $H(k,0)=0$, we conclude $H(k,t)>0$ for $t>0$. The last statement follows from Remark \ref{rem:def:elliptic_integral}.
\end{proof}

\begin{theorem}[Symmetric free boundary unduloids in the unit ball] \label{thm:unuloid}
	Every undulary symmetric with respect to the $y$-axis that meets the unit circle orthogonally falls into one of the following two cases. Recall that by Proposition \ref{prp:undulary:intersection}, such undularies are characterized by the zeros of $F^\pm$ and let $\kappa_0\in(0,1)$ be the unique number with $2E(\kappa_0)=K(\kappa_0)$ according to \eqref{eq:lem:unduloid:2E=K}.
	\begin{enumerate} \upshape
		\item Let $l$ be a nonnegative even integer. Then, there exists a unique $k_l\in(0,\kappa_0)$ such that $F^+(k_l,l\pi + \frac{\pi}{2})=0$. Moreover, for all $k\in(k_l,1)$ there exist unique $t_1\in (l\pi,l\pi+\frac{\pi}{2})$ and $t_2\in(l\pi+\frac{\pi}{2},(l+1)\pi)$ with $F^+(k,t_i)=0$ for $i=1,2$.
		\item Let $l$ be a positive odd integer. Then, there exists a unique $k_l\in(0,\kappa_0)$ such that $F^-(k_l,l\pi + \frac{\pi}{2})=0$. Moreover, for all $k\in(k_l,1)$ there exist unique $t_1\in (l\pi,l\pi+\frac{\pi}{2})$ and $t_2\in(l\pi+\frac{\pi}{2},(l+1)\pi)$ with $F^-(k,t_i)=0$ for $i=1,2$.
	\end{enumerate}
\end{theorem}
\begin{proof}
	By \eqref{eq:lem:unduloid:partialF}, the critical points of $F(k,\cdot)$ on $(0,\infty)$ are precisely the zeros of cosine and every critical point is a strict extrema. Moreover, for any nonnegative integer $l$, 
	\begin{align*}
		F^+(k,l\pi)&=-(1-k(-1)^l)\sqrt{1-k^2}<0,\\
		F^-(k,l\pi)&=(1+(-1)^lk)\sqrt{1-k^2}>0. 
	\end{align*}
	In view of \eqref{eq:prp:undulary:intersection:sign}, it will thus be enough to analyze $h^\pm$. We have
	\begin{equation*}
		h_l^+(0)=-1<0,\quad h_l^+(1)\ge0\qquad\text{for $l$ even}
	\end{equation*} 
	and 
	\begin{equation*}
		h_l^-(0)=1>0,\quad h_l^-(1)<0\qquad\text{for $l$ odd}.
	\end{equation*} 
	By \eqref{eq:lem:unduloid:2E=K}, ${h_l^\pm}'$ has a unique zero $\kappa_0\in(0,1)$ and \eqref{eq:lem:unduloid:h-pprime} implies
	\begin{equation*}
		{h_l^\pm}''(\kappa_0)=(-1)^{l+1}(2l+1)\frac{E(\kappa_0)}{\kappa_0(1-\kappa_0^2)}.
	\end{equation*}
	Thus, $\kappa_0$ is a strict maximum for $h^+_l$ if $l$ is even and a strict minimum for $h^-_l$ if $l$ is odd. This implies the existence of a unique $k_l\in(0,\kappa_0)$ with $h_l^+(k_l)=0$ if $l$ is even and $h_l^-(k_l)=0$ if $l$ is odd. Now, the conclusion follows.
\end{proof}

\begin{definition}
Let $l \geq 0$. Given $k \in (k_l, 1)$, let $t_1(k) \in (l\pi, l\pi+\frac{\pi}{2})$ and $t_2(k) \in (l\pi+\frac{\pi}{2},(l+1)\pi)$ be the unique numbers with $F^{\pm}(k, t_i(k))=0$ given by Theorem~\ref{thm:unuloid}, where we consider + if $l$ is a nonnegative even integer and - if $l$ is a positive odd integer. Define
\begin{equation}\label{def:tau:1}
	\tau_1(k) := t_1(k) - l\pi \in \left(0, \frac{\pi}{2}\right)
\end{equation}
and
\begin{equation}\label{def:tau:2}
	\tau_2(k) := (l+1)\pi - t_2(k)  \in \left(0, \frac{\pi}{2}\right).
\end{equation}
Let $\sigma_i := {(-1)}^{i+1}$.
Then, $a^\pm(k,t_i(k))$ and $F^\pm(k,t_i(k)) = 0$ simplify to
\begin{equation}\label{def:a_tau}
a_i(k, \tau_i(k)) := a^{+}(k, \tau_i(k)) = \frac{k \sin \tau_i(k)}{\sqrt{1-k^2}(1 - \sigma_i k \cos \tau_i(k))} = a^{\pm}(k, t_i(k)),
\end{equation}
\begin{equation}\label{def:F_tau}
0 = F_i(k, \tau_i(k)),
\end{equation}
where 
\begin{align*}
	F_i(k,t):=k \sin(t) I_i(k, t) - (1-\sigma_i k \cos t) \sqrt{1-k^2 \cos^2 t},\\
	I_1(k, t) := I(k, t+l\pi), \quad  I_2(k, t) :=I(k, (l+1)\pi-t).
\end{align*}
Note that by~\eqref{def:F_tau},
\begin{equation}\label{def:a_tau:simplified}
	a_i(k, \tau_i(k)) = \frac{\sqrt{1-k^2 \cos^2(\tau_i(k))}}{\sqrt{1-k^2}I_i(k, \tau_i(k))}.
\end{equation}
Furthermore, define
\begin{equation}\label{def:a_l}
	a_l := a^\pm\left(k_l, l\pi + \frac{\pi}{2}\right) = \frac{k_l}{\sqrt{1-k_l^2}}.
\end{equation}

\end{definition}
\begin{remark}
	A numerical analysis  gives:
	\begin{center}
	\begin{tabular}{ c c }
		$k_0 \approx 0.76589502$ & $a_0 \approx 1.19119122$ \\ 
		$k_1 \approx 0.21470270$ & $a_1 \approx 0.21982924$ \\  
		$k_2 \approx 0.12784799$ & $a_2 \approx 0.12890582$   
	\end{tabular}
	\end{center}
\end{remark}

\begin{lemma}\label{lem:a_l}
	The following hold:
	\begin{enumerate} \upshape
		\item\label{lem:a_l:1} $a_0 > 1$ 
		\item\label{lem:a_l:2} $a_l < \frac{1}{2(l+1)}$ for $l \geq 1$
	\end{enumerate}
\end{lemma}
\begin{proof}
	Recall that by~\eqref{eq:lem:unduloid:h}, $k_l \in (0,\kappa_0)$ is characterized by the equation 
	\begin{equation}\label{eq:k_l}
	k_l(2l+1)E(k_l) = 1.
	\end{equation}
	For $l=0$, we have
	\begin{equation*}
		E(k_0) < \int_0^{\frac{\pi}{2}} 1-\frac{1}{2}k_0^2\cos^2(x) \, \d x = \frac{\pi}{2}\left(1-\frac{k_0^2}{4} \right).
	\end{equation*}
	Thus, $1 < k_0 \frac{\pi}{2}\left(1-\frac{k_0^2}{4} \right)$. Since $k \mapsto k\left(1-\frac{k^2}{4} \right)$ has a maximum of $\frac{7}{8\sqrt{2}} < \frac{2}{\pi}$ on $[0, \frac{1}{\sqrt{2}}]$, it follows that $k_0 > \frac{1}{\sqrt{2}}$, which, by monotonicity of $k \mapsto \frac{k}{\sqrt{1-k^2}}$, implies
	\begin{equation*}
		a_0 > \frac{\frac{1}{\sqrt{2}}}{\sqrt{1-1/2}} = 1.
	\end{equation*}
	Now, let $l \geq 1$. Since $a_l = \frac{k_l}{\sqrt{1-k_l^2}}$, the inequality \eqref{lem:a_l:2} is equivalent to 
	\begin{equation*}
		k_l < \frac{1}{\sqrt{4(l+1)^2+1}} =: k_l^*.
	\end{equation*}
	Since $k \mapsto k(2l+1)E(k)$ is strictly increasing on $(0,k_0)$ and $k_l^* < \frac{1}{\sqrt{2}} < k_0<\kappa_0$, by \eqref{eq:k_l}, it suffices to prove that $1 < k_l^*(2l+1)E(k_l^*)$.
	We have
	\begin{equation*}
		E(k_l^*) > \int_0^{\frac{\pi}{2}} 1-{(k_l^*)}^2\cos^2(x) \, \d x = \frac{\pi}{2}\left(1-\frac{{(k_l^*)}^2}{2}\right) = \frac{\pi}{2} \frac{8{(l+1)}^2+1}{8{(l+1)}^2+2}.
	\end{equation*}
	Let $n:=2l+1$. Then $4{(l+1)}^2 = n^2 + 2n + 1$, $n \geq 3$ and we get
	\begin{equation*}
		k_l^*(2l+1)E(k_l^*) > \frac{n}{\sqrt{n^2+2n+2}} \frac{\pi}{2} \frac{2n^2+4n+3}{2n^2+4n+4} \geq \frac{3}{\sqrt{17}} \frac{\pi}{2}\frac{33}{34} > 1
	\end{equation*}
	which completes the proof.
\end{proof}

\begin{lemma}\label{lem:a_bounds}
	The following hold:
	\begin{enumerate} \upshape
		\item\label{lem:a_bounds:1} $\lim_{k\nearrow 1}a_i(k, \tau_i(k))=\infty$ if $l=0$
		\item\label{lem:a_bounds:2} $\lim_{k\nearrow 1}a_1(k, \tau_1(k))=\frac{1}{2l}$ for $l \geq 1$
		\item\label{lem:a_bounds:3} $\lim_{k\nearrow 1}a_2(k, \tau_2(k))=\frac{1}{2\sqrt{l(l+2)}}$ for $l \geq 1$
		\item\label{lem:a_bounds:4} $a_2(k, \tau_2(k)) < \frac{1}{\sqrt{5}}$ for all $l\geq 1$
		\item\label{lem:a_bounds:5} $a_2(k, \tau_2(k)) > \frac{1}{2}$ if $l=0$ 
	\end{enumerate}
\end{lemma}
\begin{proof}
	Rearranging~\eqref{def:F_tau} and squaring yields
	\begin{equation*}
	\frac{k^2 \sin^2(\tau_i(k))}{1-k^2 \cos^2(\tau_i(k))} = \left(\frac{1-\sigma_i k \cos(\tau_i(k))}{I_i(k, \tau_i(k))}\right)^2 =: f_i(k)
	\end{equation*}
	Using
	\begin{equation*}
		\frac{1-k^2\cos^2 \tau_i(k)}{1-k^2} = \frac{1}{1-f_i(k)},
	\end{equation*}
	and~\eqref{def:a_tau:simplified}, we conclude 
	\begin{equation}\label{eq:a_squared}
		\begin{aligned}
			{a_i(k, \tau_i(k))}^2 
			&= \frac{1}{(1-f_i(k)){I_i(k, \tau_i(k))}^2} \\
			&= \frac{1}{{I_i(k,\tau_i(k))}^2- {(1- \sigma_i k \cos(\tau_i(k)))}^2}.
		\end{aligned}
	\end{equation}
	A simple calculation shows that $(k,t) \mapsto {I_i(k,t)}^2- {(1- \sigma_i k \cos(t))}^2$ converges uniformly to $0$ as $k \nearrow 1$ on the domain $[0,1]\times [0,\frac{\pi}{2}]$. Therefore,~\eqref{lem:a_bounds:1} is proven.

	To see~\eqref{lem:a_bounds:2} and~\eqref{lem:a_bounds:3}, let $\tilde{\tau_i} \in [0, \frac{\pi}{2}]$ be an accumulation point of $\tau_i(k)$ as $k \nearrow 1$. Write $\tau_i(k_i^{(n)}) \to \tilde{\tau_i}$, $k_i^{(n)} \xrightarrow{n \to \infty} 1$. Then,
	\begin{equation*}
		I_i \left(k_i^{(n)}, \tau(k_i^{(n)}) \right) \xrightarrow{n \to \infty} 2l + 1 - \sigma_i \cos(\tilde{\tau_i}).
	\end{equation*}
	Therefore,
	\begin{equation*}
		0 = F_i\left(k_i^{(n)}, \tau(k_i^{(n)})\right) \xrightarrow{n \to \infty} 2l\sin(\tilde{\tau_i})
	\end{equation*}
	Since $l \geq 1$, we get $\tilde{\tau_i}=0$ and thus, $\lim_{k \nearrow 1} \tau_i(k) = 0$.
	Together with~\eqref{eq:a_squared}, it follows that
	\begin{equation*}
		{a_i(k, \tau_i(k))}^2 \xrightarrow{k \nearrow 1} \frac{1}{4l^2+4l(1-\sigma_i)},
	\end{equation*}
	which proves~\eqref{lem:a_bounds:2} and~\eqref{lem:a_bounds:3}.\\
	Note that by~\eqref{eq:a_squared}, the inequality~\eqref{lem:a_bounds:4} is equivalent to
	\begin{equation*}
		{I_2(k,\tau_2(k))}^2 - {(1+k\cos(\tau_2(k)))}^2 > 5.
	\end{equation*}
	But this is true since
	\begin{equation*}
		I_2(k,\tau_2(k)) = 2(l+1)E(k) - I(k,\tau_2(k)) > (2l+1)E(k) > 3
	\end{equation*}
	and $1+k\cos(\tau_2(k)) < 2$. This proves~\eqref{lem:a_bounds:4}\\
	Finally, let $k' := \sqrt{1-k^2}$. By~\cite[900.07]{ByrdFriedman1971Springer},
	\begin{equation*}
		E(k) < \frac{\pi}{2}\left( 1-\frac{k^2}{4} - \frac{3k^4}{64} \right).
	\end{equation*}
	One readily checks that this proves $\frac{3}{4}E(\frac{3}{4}) < 1$ and consequently, by~\eqref{eq:k_l} and strict monotonicity of $kE(k)$ on $(0,k_0)$, it implies that $k_0 > \frac{3}{4}$. By convexity of $k \mapsto \sqrt{\cos^2 x+k^2\sin^2 x}$, we have
	\begin{equation*}
		\sqrt{\cos^2 x+k'^2\sin^2 x} \le (1-k')\cos x+k'
	\end{equation*}
	and hence,
	\begin{equation*}
		E(k) \le (1-k')\int_0^{\pi/2}\cos x \, \d x +k'\int_0^{\pi/2} \d x =1+k'\left(\frac{\pi}{2}-1\right).
	\end{equation*}
	Since $I(k,t) > k't$, $l=0$ implies
	\begin{equation*}
		I_2(k,\tau_2(k)) = 2E(k)-I(k,\tau_2(k)) < 2+k'(\pi-2)-k't.
	\end{equation*}
	Let
	\begin{equation*}
		A(t):= \frac{k\sin t}{k'(1+k\cos t)}.
	\end{equation*}
	Then, $A'(t) >0$ for $t \in [0,\frac{\pi}{2}]$ and since $A(0)=0, A(\frac{\pi}{2}) = \frac{k}{k'} > \frac{1}{2}$ for $k>k_0>\frac34$, there exists a unique $t_k \in (0,\frac{\pi}{2})$, sucht that $A(t_k)=\frac{1}{2}$. Equivalently,
	\begin{equation}\label{eq:t_k}
		2k\sin(t_k)=k'(1+k\cos(t_k)).
	\end{equation}
	Let $y:=k\cos(t_k)$. Then, squaring and rearranging~\eqref{eq:t_k} yields
	\begin{equation*}
		k^2 = \frac{5y^2+2y+1}{y^2+2y+5} =: \Psi(y).
	\end{equation*}
	Then,
	\begin{equation*}
		\Psi'(y)=\frac{8(y^2+6y+1)}{{(y^2+2y+5)}^2} > 0 \quad \text{for } y\geq 0
	\end{equation*}
	and $y > \frac{1}{2}$, since otherwise, by monotonicity,
	\begin{equation*}
		k^2 = \Psi(y) \leq \Psi\left(\frac{1}{2}\right) = \frac{13}{25} < {\left(\frac{3}{4}\right)}^2 < k_0^2 < k^2.
	\end{equation*}
	Combining this with~\eqref{eq:t_k} gives
	\begin{equation*}
		\frac{\sqrt{1-k^2\cos^2 t_k}}{k'} = \sqrt{1+\frac{(1+y)^2}{4}} > \frac{5}{4}
	\end{equation*}
	and
	\begin{equation*}
		t_k > \sin t_k = \frac{k'(1+y)}{2k} > \frac{3}{4} \frac{k'}{k}.
	\end{equation*}
	Together, these two inequalities yield
	\begin{align*}
		\frac{2\sqrt{1-k^2\cos^2 t_k}}{k'} - I_2(k,t_k) &> 2\left( \frac{\sqrt{1-k^2\cos^2 t_k}}{k'} - 1\right)-k'(\pi - 2) + k' t_k\\
		&> \frac{1}{2}+ \frac{3{(k')}^2}{4k} - k'(\pi - 2)\\
		&\geq k'\sqrt{\frac{3}{2k}} - k'(\pi-2) \quad \text{(AM-GM)}\\
		&> 0.
	\end{align*}
	Thus, by~\eqref{eq:t_k},
	\begin{equation*}
		F_2(k,t_k) = (1+y)\left( \frac{k'}{2} I_2(k,t_k) - \sqrt{1-k^2\cos^2 t_k}\right) < 0.
	\end{equation*}
	Finally, since 
	\begin{equation*}
		\partial_t F_2(k,\tau_2(k)) = \frac{{(k')}^2(1+k\cos\tau_2(k))\cos\tau_2(k)}{\sin\tau_2(k)\sqrt{1-k^2\cos^2\tau_2(k)}} > 0,
	\end{equation*}
	and $F_2(k,\tau_2(k)) = 0$, we conclude that $\tau_2(k)>t_k$. Since $A$ was strictly increasing and $A(t_k) = \frac{1}{2}$, the claim follows.
\end{proof}

\begin{lemma}\label{lem:a_monotone}
	Let $l \geq 0$. Then, using the definitions~\eqref{def:tau:1} and~\eqref{def:a_tau}, $k \mapsto \tau_1(k)$ is strictly monotonically decreasing and $k \mapsto a_1(k, \tau_1(k))$ is strictly monotonically increasing on $(k_l, 1)$.
\end{lemma}
\begin{proof}
	Rearranging~\eqref{def:F_tau} yields
	\begin{equation}\label{eq:I_tau}
		I_1(k, \tau_1(k)) = \frac{(1-k \cos(\tau_1(k)))\sqrt{1-k^2 \cos^2(\tau_1(k))}}{k \sin(\tau_1(k))}
	\end{equation}
	and the derivatives of $I_1$ are given by
	\begin{equation}\label{eq:I_diff}
		\partial_t I(k,t) = \sqrt{1-k^2\cos^2(t)}, \quad \partial_k I(k,t) = \int_0^t \frac{-k\cos^2(x)}{\sqrt{1-k^2\cos^2(x)}} \, \d x < 0.
	\end{equation}
	For simplicity, write $I_1(t):= I_1(k,t)$, $I_k := \partial_k I_1(k,\tau_1(k))$, $a(k, t) := a_1(k,t)$, and $\tau := \tau_1(k)$ from now on. By substituting~\eqref{eq:I_tau}, we can express the derivatives of $F_1$ at $t=\tau_1(k)$ as
	\begin{equation}\label{eq:F_t}
		F_t := \partial_t F_1(k,\tau_1(k)) = \frac{\cos(\tau) (1-k\cos(\tau))(1-k^2)}{\sin(\tau)\sqrt{1-k^2 \cos^2(\tau)}}
	\end{equation}
	\begin{equation}\label{eq:F_k}	
		F_k :=\partial_k F_1(k, \tau_1(k)) = \frac{1-k^3\cos^3(\tau)}{k\sqrt{1-k^2\cos^2(\tau)}} + k\sin(\tau) I_k
	\end{equation}
	Then
	\begin{align*}
		\frac{\d \ln(a(k, \tau))}{\d k} &= \frac{1}{k} + \frac{\cos \tau}{\sin \tau}\tau' + \frac{k}{1-k^2} - \frac{-\cos(\tau) + k \sin(\tau)\tau'}{1-k \cos \tau} \\
		&= \frac{1-k^3 \cos \tau}{k(1-k^2)(1-k \cos \tau)} + \tau' \frac{\cos \tau - k}{\sin(\tau)(1-k \cos \tau)}
	\end{align*}
	Note that by the implicit function theorem, $\tau' := \tau'(k) = -\frac{F_k}{F_t}$.\\
	If $k \leq \cos \tau$, then substituting~\eqref{eq:F_t},~\eqref{eq:F_k} and $\tau'$ yields
	\begin{equation*}
		\frac{\d\ln(a(k, \tau))}{\d k} = \frac{\sqrt{1-k^2\cos^2\tau}\sin\tau}{\cos\tau(1-k^2){(1-k\cos\tau)}^2} \left(\sqrt{\frac{1-k\cos\tau}{1+k\cos\tau}} - kI_k(\cos\tau-k)\right)
	\end{equation*}
	and since $I_k < 0$, it follows that $\frac{\d\ln(a(k, \tau))}{\d k} > 0$.\\
	On the other hand, if $k > \cos \tau$, it suffices to show that $\tau' < 0$ in order to prove strict monotonicity. To complete the proof of the Lemma, we prove $\tau' < 0$ on $(k_l, 1)$. Since $F_t>0$, all that remains is to prove $F_k > 0$. We have
	\begin{equation*}
		kI_k = \int_0^{\tau+l\pi} \frac{-k^2\cos^2(x)}{\sqrt{1-k^2\cos^2(x)}} \, \d x = I_1(\tau) -\int_0^{\tau+l\pi} \frac{1}{\sqrt{1-k^2\cos^2 x}} \, \d x.
	\end{equation*}
	Define
	\begin{equation*}
		J(t) := \int_0^{t+l\pi} \frac{1}{\sqrt{1-k^2\cos^2 x}} \, \d x.
	\end{equation*}
	Then, substituting $kI_k$ and using~\eqref{eq:I_tau} two times, it follows that
	\begin{equation}\label{eq:F_k:2}
	\begin{aligned}
		F_k &= \frac{1-k^3\cos^3(\tau)}{k\sqrt{1-k^2\cos^2(\tau)}} + \sin(\tau) I_1(\tau) -\sin(\tau)J(\tau)\\
		&= \frac{(1-k\cos\tau)(2+k\cos\tau)}{k\sqrt{1-k^2\cos^2(\tau)}} -\sin(\tau) J(\tau)\\
		&= \frac{1-k\cos\tau}{k\sqrt{1-k^2\cos^2(\tau)}} \left( 2+k\cos\tau - (1-k^2\cos^2\tau)\frac{J(\tau)}{I_1(\tau)}\right).
	\end{aligned}
	\end{equation}
	
	Let 
	\begin{equation*}
		\begin{aligned}
			P(k,t) &:= (2+k\cos t)I_1(t)-(1-k^2\cos^2 t)J(t),\\
			W(k,t) &:= P(k,t) + \frac{2k\sin t}{1-k^2}F_1(k,t).
		\end{aligned}
	\end{equation*}
	Note that
	\begin{equation*}
		\begin{aligned}
		\partial_t F_1(k,t) &= k\cos t I_1(t)-\frac{k^2\sin t\cos t(1-k\cos t)}{\sqrt{1-k^2\cos^2t}} \\
		&= \frac{\cos t}{\sin t}F_1(k,t) + \frac{\cos t(1-k\cos t)(1-k^2)}{\sin t \sqrt{1-k^2\cos^2 t}}.
		\end{aligned}
	\end{equation*}
	This yields
	\begin{equation*}
		\begin{aligned}
			\partial_t P(t,k) &= -k\sin t I_1(t)+(1+k\cos t)\sqrt{1-k^2\cos^2 t} -2k^2\cos t \sin t J(t)\\
			&= \frac{2k^2\cos t\sin t}{1-k^2\cos^2 t} P(t,k) - \frac{1+4k\cos t+k^2\cos^2t}{1-k^2\cos^2 t} F_1(t,k) - \frac{2k\cos t\sqrt{1-k^2\cos^2 t}}{1+k\cos t}.
		\end{aligned}
	\end{equation*}
	Combining the above, we get
	\begin{equation*}
		\partial_t W(k,t) = \frac{2k^2\cos t\sin t}{1-k^2\cos^2 t} W(k,t) - \frac{1+k^2\cos^2 t}{1-k^2\cos^2 t} F_1(k,t).
	\end{equation*}
	Since $F_1(k,0) = -(1-k)\sqrt{1-k^2} < 0$, $F_1(k,\tau) = 0$ and $F_t >0$,
	\begin{equation*}
		F_1(k,t) < 0 \quad \text{for } 0 \leq t < \tau.
	\end{equation*}
	Set
	\begin{equation*}
		\rho(k,t) := \frac{2k^2\cos t \sin t}{1-k^2\cos^2 t}.
	\end{equation*}
	Then,
	\begin{equation}\label{eq:integrated_factor}
		\dt \left( e^{-\int_0^t \rho(k,x) \, \d x} W(k,t)\right) = -e^{-\int_0^t \rho(k,x) \, \d x} \frac{1+k^2\cos^2 t}{1-k^2\cos^2 t}F_1(k,t).
	\end{equation}
	Since
	\begin{equation*}
		W(k,0) = (2+k)I_1(0) - (1-k^2)J(0) = 2l \left[ (2+k)E(k) - (1-k^2)K(k) \right] \geq 0
	\end{equation*}
	and the right hand side in \eqref{eq:integrated_factor} is strictly positive for $0<t<\tau$, it follows
	\begin{equation*}
		P(k,\tau) = W(k,\tau) > 0.
	\end{equation*}
	With~\eqref{eq:F_k:2}, we conclude $F_k >0$, which completes the proof.
\end{proof}

\begin{theorem}\label{thm:unduloid:a}
	Given $a>0$, there exists $k>0$ such that $U^+(a,k)$ or $U^-(a,k)$ intersects the unit circle orthogonally if and only if
	\begin{equation*}
		a \in \left( 0, \frac{1}{2} \right) \cup [\beta, \infty),
	\end{equation*}
	where, with $l=0$,
	\begin{equation*}
		\beta = \min_{k \in (k_0, 1)} a_2(k,\tau_2(k)) \in \left(\frac{1}{2}, a_0\right).
	\end{equation*}
\end{theorem}
\begin{remark}
	A numerical analysis shows that $\beta \approx 1.01783$.
\end{remark}
\begin{proof}
	Since $F_i(k, \frac{\pi}{2}) = kI(k,l\pi + \frac{\pi}{2})-1 = k(2l+1)E(k)-1$, the monotonicity~\eqref{eq:lem:unduloid:partialF} implies $\lim_{k \searrow k_l} \tau_i(k) = \frac{\pi}{2}$ and hence,
	\begin{equation*}
		\lim_{k \searrow k_l} a_i(k, \tau_i(k)) = a_l.
	\end{equation*}
	Then, combining the previous lemmas, it follows that
	\begin{align*}
		(a_0, \infty) &= \{a_1(k,\tau_1(k)) | k \in (k_0,1)\},\\
		\left(a_l, \frac{1}{2l}\right) &= \{a_1(k,\tau_1(k)) | k \in (k_l,1)\} \quad \text{for } l\geq 1,\\
		\left(0, \frac{1}{2}\right) \supset \left(0, \frac{1}{\sqrt{5}}\right) &\supset \{a_2(k,\tau_2(k)) | k \in (k_l,1)\} \quad \text{for } l\geq 1.
	\end{align*}
	Since $a_1(k,\tau_1(k)) > a_0$ if $l=0$ and $a_1(k,\tau_1(k)) < \frac{1}{2(l+1)}$ for $l\geq 1$, it follows that a solution for $a>0$ exists if and only if
	\begin{equation*}
		a \in \left( 0, \frac{1}{2} \right) \cup \left[\beta, \infty\right),
	\end{equation*}
	provided the infimum $\beta  = \inf_{k \in (k_0, 1)} a_2(k,\tau_2(k))$ with $l=0$ is attained.\\
	To show that the infimum is attained, we will show that
	\begin{equation*}
		\lim_{k \searrow k_0} \frac{\d}{\d k} a_2(k,\tau_2(k)) < 0.
	\end{equation*}
	With a change of variables, it can be shown that
	\begin{equation*}
		I(c) := I_2(k,\arccos c) = E(k) + \int_0^c \frac{\sqrt{1-k^2u^2}}{\sqrt{1-u^2}} \, \d u
	\end{equation*}
	and
	\begin{equation*}
		G(k,c) := F_2(k,\arccos c) = k\sqrt{1-c^2} I(c) - (1+kc)\sqrt{1-k^2c^2}.
	\end{equation*}
	Then,
	\begin{equation*}
		\partial_k G(k_0,0) = \left.\frac{\d}{\d k} (kE(k))\right\vert{}_{k_0} > 0
	\end{equation*}
	and
	\begin{equation*}
		\partial_c G(k_0,0) = 0, \quad \partial_c^2 G(k_0,0) = -(1-k_0^2) < 0.
	\end{equation*}
	Using these results, the implicit function theorem proves the existence of a smooth function $k(c)$ with $G(k(c),c) = 0$ and
	\begin{align*}
		k(0) = k_0, \quad k'(0)=0, \quad k''(0) = -\frac{\partial_c^2 G(k_0,0)}{\partial_k G(k_0,0)} > 0.
	\end{align*}
	Hence,
	\begin{equation*}
		k'(c) = k''(0)c + \mathcal{O}(c^2) > 0
	\end{equation*}
	for sufficiently small $c>0$. Let $A(k):=a_2(k,\tau_2(k))$ and $B(c):=A(k(c))$. Then,
	\begin{equation*}
		B(c) = \frac{k(c)\sqrt{1-c^2}}{\sqrt{1-{k(c)}^2}\big(1+k(c)c\big)}
	\end{equation*}
	and, since $k'(0)=0$,
	\begin{equation*}
		B'(0) = -\frac{k_0^2}{\sqrt{1-k_0^2}} < 0.
	\end{equation*}
	Therefore, $B'(c) < 0$ for sufficiently small $c>0$ and the chain rule gives
	\begin{equation*}
		A'(k(c)) = \frac{B'(c)}{k'(c)} < 0
	\end{equation*}
	which implies the conclusion.
\end{proof}

\subsection{Free boundary nodoids}
We will denote by 
\begin{equation}\label{eq:N}
	N^\pm(a,b)\vcentcolon=(g_\mathrm{n}^\pm,f_\mathrm{n}^\pm)
\end{equation}
the two nodaries that are symmetric with respect to the $y$-axis, see \eqref{eq:nodary:g}-\eqref{eq:nodary:f}.

\begin{proposition}\label{prp:nodary:system}
	For the nodary $N^\pm(a,b)$, the system \eqref{it:circle}-\eqref{it:tangent} is equivalent to
	\begin{equation*}
		g^\pm(t)\sin(t)=f^\pm(t)\cos(t),\qquad f^\pm(t)=|\sin(t)|. 
	\end{equation*}
\end{proposition}

\begin{proof}
	This is a consequence of Proposition~\ref{prp:nodary:derivatives}.
\end{proof}

\begin{proposition}[Characterizing orthogonal intersections of nodary and circle]\label{prp:nodary:intersection}
    Given $a,t>0$, there exists $b>0$ such that $N^\pm(a,b)$ intersects the unit circle orthogonally at $t$ if and only if 
    \begin{equation}\label{eq:prp:nodary:intersection:a}
    	\sin(t)\neq 0,\qquad|\sin(t)|\pm a \cos(t)>0
    \end{equation}
	and, denoting with $\sigma(t)$ the sign of $\sin(t)$,
    \begin{align*}
        0 &=  \sigma(t)\cos(t) \mp a\sin(t) + \int_0^t \frac{a^2\cos^2(x) \, \mr d x}{\sqrt{\sin^2(t) \pm 2a|\sin(t)|\cos(t)+a^2\cos^2(x)}}\\
        &=\vcentcolon G^\pm(a,t).
    \end{align*}
	In this case,
	\begin{equation}\label{eq:prp:nodary:intersection:b}
		b=\sqrt{\sin^2(t)\pm 2a|\sin(t)|\cos(t)}.
	\end{equation}
\end{proposition}

\begin{proof}
	Since $f^\pm>0$, squaring the second equation of Proposition~\ref{prp:nodary:system} implies \eqref{eq:prp:nodary:intersection:a} and \eqref{eq:prp:nodary:intersection:b}. Propositions \ref{prp:nodary:system} and \ref{prp:nodary:derivatives}\,\eqref{it:prp:nodary:derivatives} together with \eqref{eq:prp:nodary:intersection:b} imply $0=G^\pm(a,t)$.
\end{proof}

\begin{lemma}\label{lem:nodoid}
	Let $G^\pm$ be as in Proposition~\ref{prp:nodary:intersection}. For all $a>0$ and $t\in(0,\pi)$, define
	\begin{align*}
		A^\pm(a,t)&\vcentcolon=\sin(t)\pm a\cos(t),\\
		B^\pm(a,t)&\vcentcolon=\sin^2(t)\pm 2a\sin(t)\cos(t),\\ 
		C(a,t)&\vcentcolon=a^2\cos^2(t)
	\end{align*}	
	and, for all positive integers $n$,
	\begin{equation*}
		I_n^\pm(a,t)\vcentcolon=\int_0^t\frac{C(a,x)\,\mr dx}{(B^\pm(a,t)+C(a,x))^\frac{n}{2}}.
	\end{equation*}
	Then,
	\begin{align*}
		\partial_tA^\pm(a,t)&=\cos(t)\mp a\sin(t),\\
		\partial_t^2A^\pm(a,t)&=-A^\pm(a,t), \\
		\partial_tB^\pm(a,t)&=2\sin(t)\cos(t)\pm 2a\bigl(\cos^2(t)-\sin^2(t)\bigr)\\ 
		\partial_t^2B^\pm(a,t)&=2\bigl(\cos^2(t)-\sin^2(t)\bigr)\mp8a\sin(t)\cos(t),
	\end{align*}
	and
	\begin{align} \nonumber
		G^\pm&=\partial_t A^\pm+I_1^\pm, \\ \label{eq:lem:nodoid:partialG}
		\partial_tG^\pm&=-A^\pm + \frac{C}{A^\pm} -\frac{1}{2}\partial_tB^\pm I_3^\pm=-\frac{B^\pm}{A^\pm}-\frac{1}{2}\partial_tB^\pm I_3^\pm, \\ \label{eq:lem:nodoid:ppartialG-no_frac} 
		\partial_t^2G^\pm &=-\partial_tA^\pm + \frac{\partial_tC}{A^\pm} - \frac{C\partial_tA^\pm}{(A^\pm)^2}-\frac{1}{2}\partial_t^2B^\pm I_3^\pm-\frac{C\partial_tB^\pm}{2(A^\pm)^3} + \frac{3}{4}(\partial_tB^\pm)^2I_5^\pm \\ \label{eq:lem:nodoid:ppartialG}
		&=-\frac{\partial_t B^\pm}{A^\pm} + \frac{B^\pm\partial_tA^\pm}{(A^\pm)^2}-\frac{1}{2}\partial_t^2B^\pm I_3^\pm-\frac{C\partial_tB^\pm}{2(A^\pm)^3} + \frac{3}{4}(\partial_tB^\pm)^2I_5^\pm.
	\end{align}
\end{lemma}

\begin{proof}
	This is a straight forward computation.
\end{proof}

\begin{lemma}\label{lem:nodoid:plus}
	Let $G^+$ be as in Proposition \ref{prp:nodary:intersection} and for all $a>0$, set
	\begin{equation*}
		g(a)\vcentcolon=G^+\Bigl(a,\frac{\pi}{2}\Bigr).
	\end{equation*}
	Then,
	\begin{equation}\label{eq:lem:nodoid:plus:g}
		g(a)=-a - \frac{K(\frac{a}{\sqrt{1+a^2}})}{\sqrt{1+a^2}}+\sqrt{1+a^2}E\Bigl(\frac{a}{\sqrt{1+a^2}}\Bigr)
	\end{equation}
	and $g(a)<0$.
\end{lemma}

\begin{proof}
	We have
	\begin{align*}
		g(a)&=-a+\int_0^\frac{\pi}{2}\frac{a^2\cos^2(x)\,\mr dx}{\sqrt{1+a^2\cos^2(x)}}\\
		&=-a-\int_0^\frac{\pi}{2}\frac{\mr dx}{\sqrt{1+a^2-a^2\sin^2(x)}} + \int_0^\frac{\pi}{2}\frac{1+a^2-a^2\sin^2(x)}{\sqrt{1+a^2-a^2\sin^2(x)}}\mr dx
	\end{align*}
	which by Definition \ref{def:elliptic_integral} implies \eqref{eq:lem:nodoid:plus:g}. Recalling Remark \ref{rem:def:elliptic_integral}, it is a straight forward computation to see that
	\begin{equation*}
		g'(a)=-1+kE(k)\qquad\text{for $k=\frac{a}{1+a^2}$}.
	\end{equation*}
	Notice that $g'$ as a function of $k$ coincides with $h_0^+$ from \eqref{eq:lem:unduloid:h}. Using \eqref{eq:lem:unduloid:h-prime}, \eqref{eq:lem:unduloid:h-pprime}, and \eqref{eq:lem:unduloid:2E=K}, we have already seen that $g'$ has a unique zero at some $a_0>0$ at which $g$ attains a strict minimum. An elementary computation shows that
	\begin{equation*}
		\lim_{a\searrow0}g(a)=0,\qquad \lim_{a\nearrow\infty}g(a)=0. 
	\end{equation*}
	It follows that $g$ is strictly negative.
\end{proof}

\begin{theorem}[Symmetric free boundary nodoids in the unit ball]\label{thm:nodoid:plus}
	For all $a>0$, there exists a unique $t_0\in(0,\pi/2)$ such that for $b=(\sin^2(t_0)+2a\sin(t_0)\cos(t_0))^{1/2}$, the nodary $N^+(a,b)$ as defined in \eqref{eq:N} restricted to the interval $(-t_0,t_0)$ lies inside the unit ball and meets the unit sphere orthogonally at $t_0$.
\end{theorem}

\begin{proof}
	Since $\lim_{t\searrow0}G^+(a,t)=1$, existence is a consequence of Proposition~\ref{prp:nodary:intersection}, Lemma~\ref{lem:nodoid:plus}, and the intermediate value theorem. To prove uniqueness, it will first be shown that every critical point of $G^+(a,\cdot)$ on $(0,\pi/2)$ is a strict minimum. To this end, assume that $G^+(a,\cdot)$ has a critical point at $t_\mr{c}\in(0,\pi/2)$. Using the notation of Lemma~\ref{lem:nodoid}, we abbreviate
	\begin{gather*}
		A\vcentcolon=A^+(a,t_\mr{c}),\qquad A'\vcentcolon=\partial_tA^+(a,t_\mr{c}),\qquad B\vcentcolon=B^+(a,t_\mr{c}), \\
		B'\vcentcolon=\partial_tB^+(a,t_\mr{c}),\qquad B''\vcentcolon=\partial_t^2B^+(a,t_\mr{c}),\qquad C\vcentcolon=C(a,t_\mr{c}),\\ 
		I_3\vcentcolon=I_3^+(a,t_\mr{c}), \qquad I_5\vcentcolon=I_5^+(a,t_\mr{c}),\qquad G''\vcentcolon=\partial_t^2G^+(a,t_\mr{c}).
	\end{gather*}
	In view of \eqref{eq:lem:nodoid:partialG}, $t_\mr{c}$ is characterized by the equation
	\begin{equation}\label{eq:proof:thm:nodoid:plus}
		\frac{B}{A}=-\frac{1}{2}B'I_3.
	\end{equation}
	Let
	\begin{equation*}
		P\vcentcolon=A^2(B')^2-AA'BB'-A^2BB''+\frac{1}{2}(B')^2C.
	\end{equation*}
	Then, plugging \eqref{eq:proof:thm:nodoid:plus} into \eqref{eq:lem:nodoid:ppartialG}, we infer
	\begin{align*}
		G''&=\frac{(B')^2I_3}{2B}-\frac{A'B'I_3}{2A}-\frac{B''I_3}{2}+\frac{(B')^2CI_3}{4A^2B}+\frac{3(B')^2I_5}{4} \\
		&=P\frac{I_3}{2A^2B}+\frac{3(B')^2I_5}{4}.
	\end{align*}
	Thus, in order to show that $G^+(a,\cdot)$ has a strict minimum at $t_\mr{c}$, it is enough to prove that $P>0$. Indeed, an elementary but somewhat lengthy computation gives
	\begin{align*}
		P&=2\sin^6(t_\mr{c}) + 12a\sin^5(t_\mr{c})\cos(t_\mr{c})\\
		&\quad+2a^2\sin^2(t_\mr{c})\bigl(\sin^4(t_\mr{c})+16\sin^2(t_\mr{c})\cos^2(t_\mr{c})+3\cos^4(t_\mr{c})\bigr)\\
		&\quad+2a^3\sin(t_\mr{c})\cos(t_\mr{c})\bigl(\sin^4(t_\mr{c})+15\sin^2(t_\mr{c})\cos^2(t_\mr{c})+6\cos^4(t_\mr{c})\bigr)\\
		&\quad+2a^4\cos^2(t_\mr{c})\bigl(\sin^4(t_\mr{c})+4\sin^2(t_\mr{c})\cos^2(t_\mr{c})+3\cos^4(t_\mr{c})\bigr)
	\end{align*}
	which is strictly positive since $t_\mr{c}\in(0,\pi/2)$. Hence, every zero of $\partial_tG^+(a,\cdot)$ is a strict minimum of $G^+(a,\cdot)$ and the mountain pass theorem implies that $\partial_tG^+(a,\cdot)$ has at most one zero. Since $G^+(a,0)=1$ and $G^+(a,\pi/2)<0$, the function $G^+(a,\cdot)$ has exactly one zero in $(0,\pi/2)$. 
\end{proof}

\begin{remark*}
	Note that even though we have shown $\partial_t^2G^+(a,t)>0$ at every critical point $t$ of $G^+(a,\cdot)$, there exist $a>0$ such that $G^+(a,\cdot)$ is not convex on $(\pi/4,\pi/2)$.
\end{remark*}

\begin{lemma}\label{lem:nodoid:minus}
    Let $G^-$ be as in Proposition~\ref{prp:nodary:intersection}. Then,	
    \begin{enumerate} \upshape
        \item \label{it:lem:properties_G:partial_aG} $\partial_a G^-(a, t) > 0$ for all $a>0$ and $t \in (\frac{\pi}{2}+\arctan(a),\pi)$;
        \item\label{it:lem:properties_G:ppartial_aG} $\partial_t^2 G^-(a, t) > 0$ for all $a>0$ and $t \in (\frac{\pi}{2}+\arctan(a),\pi)$;
        \item \label{it:lem:properties_G:t_0} for all $0<a<1$ there exists a unique $t_0\in (\frac{\pi}{2}+\arctan(a),\pi)$ such that $\partial_t G^-(a,t_0)=0$; 
        \item \label{it:lem:properties_G:a} for $a \geq 1$ there exists no $t_0 \in (\frac{\pi}{2}+\arctan(a), \pi)$ with $G^-(a,t_0)=0$.
    \end{enumerate}
\end{lemma}

\begin{proof}
	Let $a>0$ and $t \in (\frac{\pi}{2}+\arctan(a),\pi)$. We have 
	\begin{equation*}
		\partial_aG^-(a,t)= \int_0^t \frac{a\cos^2(x) \left( a^2\cos^2(x)+ 2\sin^2(t) - 3a\sin(t)\cos(t)\right)}{{\left(a^2\cos^2(x)+B^-(t)\right)}^{\frac{3}{2}}} \, \mr d x
	\end{equation*}
	which is strictly positive. Thus, \eqref{it:lem:properties_G:partial_aG} is proven. Note that for $t \in (\frac{\pi}{2}+\arctan(a),\pi)$,
	\begin{equation}\label{eq:lem:nodoid:minus:proof:signs}
		\partial_tA^-(a,t)<0,\qquad\partial_tB^-(a,t)<0,\qquad\partial_tC(a,t)>0.
	\end{equation}
	Abbreviate 
	\begin{gather*}
		B\vcentcolon=B^-(a,t),\qquad B'\vcentcolon=\partial_tB^-(a,t),\qquad B''\vcentcolon=\partial_t^2B^-(a,t),\\
		I_3\vcentcolon=I_3^-(a,t),\qquad I_5\vcentcolon=I_5^-(a,t).
	\end{gather*}
	Then, by \eqref{eq:lem:nodoid:minus:proof:signs} and \eqref{eq:lem:nodoid:ppartialG-no_frac}, in order to prove \eqref{it:lem:properties_G:ppartial_aG}, it is enough to show
	\begin{equation*}
		-\frac{1}{2}B''I_3+\frac{3}{4}(B')^2I_5\geq0.
	\end{equation*}
	Let $x\vcentcolon=\cot(t)$. It holds
	\begin{equation*}
		B = \frac{1-2ax}{1+x^2},\qquad B' = \frac{2x-2ax^2+2a}{1+x^2},\qquad B'' = \frac{2x^2-2+8ax}{1+x^2}.
	\end{equation*}
	Thus, using $I_5\ge I_3/(a^2+B)$, we compute
	\begin{align*}
		&\quad-\frac{1}{2}B''I_3+\frac{3}{4}(B')^2I_5\\
		&\ge \frac{I_3}{4(a^2+B)}\Bigl(-2B''(a^2+B)+3(B')^2\Bigr)\\
		&=\frac{I_3}{(a^2+B)(1+x^2)^2}\Bigl(1+4a^2 - 4a^3x + 2(1+a^2)x^2-4a(1+a^2)x^3+2a^2x^4\Bigr)
	\end{align*}
	which is strictly positive since $x<0$. Thus, \eqref{it:lem:properties_G:ppartial_aG} is proven. Let $\tau=\frac{\pi}{2}+\arctan(a)$. Using
	\begin{equation}\label{eq:lem:nodoid:minus:proof:sin-cos}
		\sin(\tau)=\frac{1}{\sqrt{1+a^2}},\qquad \cos(\tau)=\frac{-a}{\sqrt{1+a^2}}
	\end{equation}
	and
	\begin{equation*}
		I_3^-(a,\tau)<\frac{\pi a^2}{B^-(a,\tau)^\frac32}=\frac{\pi a^2(1+a^2)^\frac32}{(1+2a^2)^\frac32},
	\end{equation*}
	we infer 
	\begin{equation*}
		\partial_tG^-(a,\tau)=-\frac{1+2a^2}{(1+a^2)^\frac32}+\frac{a^3}{1+a^2}I_3^-(a,\tau)<-\frac{1+2a^2}{(1+a^2)^\frac32}+\frac{\pi a^5\sqrt{1+a^2}}{(1+2a^2)^\frac32}.
	\end{equation*}
	After rearranging and squaring, we see that $\partial_tG^-(a,\tau)<0$ provided
	\begin{equation*}
		\pi^2a^{10}(1+a^2)^4<(1+2a^2)^5.
	\end{equation*} 
	Indeed, for $0<a<1$, we have
	\begin{align*}
		&(1+2a^2)^5-\pi^2a^{10}(1+a^2)^4\\
		&>(1+2a^2)^5-\pi^2a^2(1+a^2)^4 \\
		&=1+a^2(10-\pi^2)+a^4(40-4\pi^2)+a^6(80-6\pi^2)+a^8(80-4\pi^2)+a^{10}(32-\pi^2)
	\end{align*}
	which is strictly positive. On the other hand, $\partial_tG^-(a,\pi)=aI_3^-(a,\pi)>0$. Thus, \eqref{it:lem:properties_G:t_0} follows from the intermediate value theorem. By \eqref{it:lem:properties_G:partial_aG}, it is enough to prove \eqref{it:lem:properties_G:a} for $a=1$. To this end, let $t\in(\frac{\pi}{2}+\arctan(1),\pi)=(\frac{3\pi}{4},\pi)$. Then,
	\begin{equation*}
		G^-(1,t) > \sin(t) + \cos(t) + \int_0^t \frac{\cos^2(x)}{\sqrt{2 - \sin(2t)}} \, \d x = \sin(t) + \cos(t) + \frac{\frac{t}{2}+\frac{\sin(2t)}{4}}{\sqrt{2 - \sin(2t)}}.
	\end{equation*}
	Since $\sin(t)+\cos(t)\le0$, we may rearrange and square to see that $G^-(1,t)>0$ provided
	\begin{equation*}
		(\sin(2t) + 2t)^2 > 16(1+\sin(2t))(2 - \sin(2t)).
	\end{equation*}
	Letting $u=\sin(2t) \in (-1,0)$ and noting $2t = 2\pi + \arcsin(u)$, this follows if
	\begin{equation*}
		(2\pi + \arcsin(u) + u)^2 > 16(1+u)(2-u).
	\end{equation*}
	Using that $\arcsin(u) \geq \frac{\pi}{2}u$ for $u \in (-1,0)$, this is indeed the case since
	\begin{equation*}
		{\left(2\pi + \Bigl(1 + \frac{\pi}{2}\Bigr)u\right)}^2 > 16(1+u)(2-u).
	\end{equation*}
	Thus, \eqref{it:lem:properties_G:a} is proven.
\end{proof}

\begin{theorem}[Symmetric free boundary nodoids outside the unit ball]\label{thm:nodoid:minus}
    \leavevmode
    Let $N^-$ be as in \eqref{eq:N}. Then, the following hold.
    \begin{enumerate}\upshape
        \item The set $I \coloneq \{ a>0 \mid \exists \, b>0\colon N^-(a,b) \text{ intersects } \mathbb{S}^1 \text{ orthogonally}\} = (0, \beta]$ for some $\frac{1}{2} < \beta < 1$.
        \item For $a \in (0, \frac{1}{2}] \cup \{\beta\}$, there exists a unique $b>0$ such that $N^-(a,b)$ intersects~$\mathbb{S}^1$ orthogonally.
        \item For $a \in (\frac{1}{2}, \beta)$, there exist exactly two values for $b>0$ such that $N^-(a,b)$ intersects $\mathbb{S}^1$ orthogonally.
        \item \label{it:thm:nodoid:minus:beta} $a = \beta$ if and only if there exists $t_0\in (\frac{\pi}{2}+\arctan(a), \pi)$ with $\partial_t G^-(a,t_0)=G^-(a,t_0)=0$.
    \end{enumerate}
\end{theorem}

\begin{remark}
    Based on the characterization of $\beta$ in \eqref{it:thm:nodoid:minus:beta}, a numerical analysis shows that $\beta \approx 0.73714$.
\end{remark}

\begin{proof}[Proof of Theorem \ref{thm:nodoid:minus}]
	Firstly, we note that by the basic properties of the nodary $t\mapsto(g_\mr{n}^-(t),f_\mr{n}^-(t))$ as defined in \eqref{eq:nodary:g}-\eqref{eq:nodary:f}, the relevant interval for the intersection point $t$ is $(\frac{\pi}{2},\pi)$. Secondly, by the definition of $G^-$ in Proposition \ref{prp:nodary:intersection}, there holds $G^-(a,t)=0$ for $a>0$ and $t\in(\frac{\pi}{2},\pi)$ only if $\cos(t)+a\sin(t)<0$ which means $t\in(\frac{\pi}{2}+\arctan(a),\pi)$. In view of \eqref{eq:lem:nodoid:minus:proof:sin-cos}, we have 
	\begin{equation*}
		G^-\Bigl(a,\frac{\pi}{2}+\arctan(a)\Bigr)>0,\qquad G^-(a,\pi) = 2a - 1.
	\end{equation*}
	Thus, the conclusion follows from Proposition~\ref{prp:nodary:intersection} and Lemma~\ref{lem:nodoid:minus}.
\end{proof}

%% file: main.bbl
\begin{thebibliography}{1}

\bibitem{BenditoBowickMedina2014JGSP}
E.~Bendito, M.~J. Bowick, and A.~Medina.
\newblock A natural parameterization of the roulettes of the conics generating
  the {D}elaunay surfaces.
\newblock {\em J. Geom. Symmetry Phys.}, 33:27--45, 2014.

\bibitem{MR3708014}
R.~G. Bettiol, P.~Piccione, and B.~Santoro.
\newblock Deformations of free boundary {CMC} hypersurfaces.
\newblock {\em J. Geom. Anal.}, 27(4):3254--3284, 2017.

\bibitem{ByrdFriedman1971Springer}
P.~F. Byrd and M.~D. Friedman.
\newblock {\em Handbook of elliptic integrals for engineers and scientists},
  volume Band 67 of {\em Die Grundlehren der mathematischen Wissenschaften}.
\newblock Springer-Verlag, New York-Heidelberg, second edition, 1971.

\bibitem{MR4979235}
A.~Cerezo, I.~Fern\'andez, and P.~Mira.
\newblock Annular solutions to the partitioning problem in a ball.
\newblock {\em J. Reine Angew. Math.}, 828:57--81, 2025.

\bibitem{Delaunay1841JMPA}
C.-E. Delaunay.
\newblock Sur la surface de r{\'e}volution dont la courbure moyenne est
  constante.
\newblock {\em Journal de math{\'e}matiques pures et appliqu{\'e}es},
  6:309--315, 1841.

\bibitem{FraserSchoen16}
A.~Fraser and R.~Schoen.
\newblock Sharp eigenvalue bounds and minimal surfaces in the ball.
\newblock {\em Invent. Math.}, 203(3):823--890, 2016.

\bibitem{MladenovaMladenov24Mathematics}
C.~D. Mladenova and I.~M. Mladenov.
\newblock Parameterizations of delaunay surfaces from scratch.
\newblock {\em Mathematics}, 12(10):1570, 2024.

\bibitem{Scharrer22NonAna}
C.~Scharrer.
\newblock Embedded {D}elaunay tori and their {W}illmore energy.
\newblock {\em Nonlinear Anal.}, 223:Paper No. 113010, 23, 2022.

\bibitem{Steklov1902}
V.~Steklov.
\newblock Sur les problèmes fondamentaux de la physique mathématique (suite
  et fin).
\newblock {\em Annales scientifiques de l'École Normale Supérieure}, 3e
  série, 19:455--490, 1902.

\end{thebibliography}
